\documentclass[a4paper,12pt]{amsart}

\usepackage[T1]{fontenc}
\usepackage[utf8]{inputenc}
\usepackage[english]{babel}

\usepackage[hyperindex,breaklinks]{hyperref}
\hypersetup{colorlinks,
linkcolor=blue,
filecolor=green,
urlcolor=blue,
citecolor=blue}

\usepackage{amsmath}
\usepackage{amssymb}
\usepackage{mathrsfs}
\usepackage{eucal}

\usepackage{amsthm}
\usepackage{enumitem}  %Options for enumerate environment. Supersedes 'enumerate'
\setlist{itemsep=\medskipamount,font=\normalfont}

\usepackage[normalem]{ulem}

\DeclareMathAlphabet{\pazocal}{OMS}{zplm}{m}{n}

\usepackage{graphicx}
\usepackage{color}
\usepackage{tikz}

\input xy
\xyoption{all}

\usepackage{multicol}

\usepackage[backend=bibtex,style=ieee,url=false,sorting=nty]{biblatex} %biber

\bibliography{library.bib}
\usepackage{csquotes}

\newtheorem{theorem}{Theorem}[section]
\newtheorem{lemma}[theorem]{Lemma}
\newtheorem{corollary}[theorem]{Corollary}
\newtheorem{proposition}[theorem]{Proposition}

\theoremstyle{definition}
\newtheorem{definition}[theorem]{Definition}
\newtheorem{example}[theorem]{Example}

\newtheorem{remark}[theorem]{Remark}
\newtheorem{conj}[theorem]{Conjecture}
\newtheorem{prob}[theorem]{Problem}

\newcommand{\Aut}{\operatorname{Aut}}
\newcommand{\Typ}{\operatorname{Typ}}
\newcommand{\typ}{\operatorname{typ}}

\newcommand{\comment}[1]{}
\newcommand{\sgn}{\operatorname{sgn}}

\def\a{\alpha}
\def\b{\beta}

\def\g{\gamma}

\def\G{\Gamma}
\def\mG{\mathcal{G}}

\newcommand{\Dom}{\operatorname{Dom}}
\newcommand{\Img}{\operatorname{Im}}

\newcommand{\gr}{\operatorname{gr}}

\newcommand{\Mod}{\operatorname{Mod}}
\newcommand{\Gr}{\operatorname{Gr}}

\newcommand{\id}{\operatorname{id}}

\title[The talented monoid of a directed graph]{The talented monoid of a directed graph\\with applications to graph algebras}

\author{Luiz Gustavo Cordeiro}
\address{Luiz Gustavo Cordeiro: Departamento de Matem\'{a}tica - UFSC - Florian\'{o}polis - SC, Brazil}
\email{luiz.cordeiro@ufsc.br}

\author{Daniel Gon\c{c}alves}
\address{Daniel Gon\c{c}alves: Departamento de Matem\'{a}tica - UFSC - Florian\'{o}polis - SC, Brazil}
\email{daemig@gmail.com}
\thanks{D. Gon\c{c}alves was partially supported by Conselho Nacional de Desenvolvimento Cient\'ifico e Tecnol\'ogico (CNPq) grant numbers 304487/2017-1 and 406122/2018-0 and Capes-PrInt grant number 88881.310538/2018-01 - Brazil.}

\author{Roozbeh Hazrat}

\address{Roozbeh Hazrat: 
Centre for Research in Mathematics and Data Science\\
Western Sydney University\\
Australia} \email{r.hazrat@westernsydney.edu.au}
\thanks{R. Hazrat acknowledges Australian Research Council grant DP160101481. Part of the research carried out while he was visiting Universidade Federal de Santa Catarina, Florian\'opolis in February 2020. He would like to thank his host Daniel Gon\c{c}alves for the warm hospitality. }

\thanks{The authors would like to express their gratitude to both referees who in total provided 8 pages of comments, suggestions and corrections which helped to improve the presentation of the paper. }

\subjclass[2010]{18B40,16D25}

\keywords{directed graph, talented monoid, graded Grothendieck group, graded ring, graph monoid, Leavitt path algebras, type semigroup}

\begin{document}
\begin{abstract}
It is a conjecture that for the class of Leavitt path algebras associated to finite directed graphs, 
their graded Grothendieck groups  $K_0^{\gr}$ are a complete invariant. For a Leavitt path algebra $L_{\mathsf k}(E)$, with coefficients in a field ${\mathsf k}$, the monoid of the positive cone of $K_0^{\gr}(L_{\mathsf k}(E))$ can be described completely in terms of the graph $E$. In this note we further investigate the structure of this  ``talented monoid'', showing how it captures intrinsic properties of the graph and hence the structure of its  associated Leavitt path algebras. More precisely, we show that the standard graph moves that give graded Morita equivalence of Leavitt path algebras also preserve the associated talented monoids and, for the class of strongly connected graphs, we show that the notion of the period of a graph can be completely described via the talented monoid.  As an application, we give a finer characterisation of the purely infinite simple Leavitt path algebras in terms of properties of the associated graph. We show that graded isomorphisms of algebras preserve the period of the graphs, and obtain results giving more evidence to support the graded classification conjecture.
 \end{abstract}

\maketitle

\section{Introduction}

\bigskip 

Let $E$ be a row-finite graph, with vertices denoted by $E^0$, edges by $E^1$, and range and source maps denoted by $r$ and $s$ respectively. The \emph{talented monoid} of the graph $E$ is defined as  
\begin{equation*}
T_E= \Big \langle \, v(i), v \in E^0, i \in \mathbb Z  \, \,  \Big \vert \, \, v(i)= \sum_{e\in s^{-1}(v)} r(e)(i+1) \, \Big \rangle,
\end{equation*}
where the relations are over vertices $v$ which are not sinks (cf.\ Definition~\ref{talentedmon}). The monoid $T_E$ is equipped with a $\mathbb Z$-action: $n\in \mathbb Z$ acts on the generators by ${}^n v(i):=v(n+i)$, and is extended to all elements of $T_E$ linearly.

The talented monoid $T_E$ can be considered as a ``time evolution model'' of the monoid $M_E$ introduced by Ara-Moreno-Pardo~\cite{MR2310414} in relation with the $K_0$-group of the Leavitt path algebra associated to $E$. For a directed graph $E$, 
\begin{equation*}
M_E= \Big \langle \, v \in E^0 \, \, \Big \vert \, \,  v= \sum_{e\in s^{-1}(v)} r(e) \, \Big \rangle,
\end{equation*}
where the relations are over vertices $v$ which are not sinks. It was proved in~\cite{MR2310414} that $M_E$ is isomorphic to the commutative monoid $\mathcal{V}(L_{\mathsf k}(E))$ of finitely generated projective $L_{\mathsf k}(E)$-modules, with direct sum as addition (i.e, non-stable $K$-theory of $L_{\mathsf k}(E)$),  where $L_{\mathsf k}(E)$ is the Leavitt path algebra of $E$ with coefficients in a field  ${\mathsf k}$. 

The first place the talented monoid $T_E$ appeared was in~\cite[Lemma~9]{MR3045160}, where it was disguised as a positive cone of the graded Grothendieck group $K_0^{\gr}(L_{\mathsf k}(E))$, and it was further investigated in~\cite{MR3781435}. In the form presented here, it was first introduced and studied in \cite{MR4040730}, where it was shown that, contrary to $M_E$, one can describe certain geometric properties of a graph $E$, such as cycles with or without exits and line points, in terms of the structure of $T_E$. For instance, it was shown that a graph has Condition~(L), i.e., any cycle has an exit, if and only if the group $\mathbb Z$ acts freely on $T_E$.

In this note, we further investigate  the structure of the talented monoid $T_E$ and provide more instances where this monoid can capture interesting  properties of the graph $E$. As a consequence, we obtain more evidence to the claim that $T_E$ is a complete invariant for graded Morita equivalence of Leavitt path algebras (Conjecture~\ref{conjehfyhtr}). More precisely, we show that the standard graph moves that give graded Morita equivalence of Leavitt path algebras also preserve the associated talented monoids and, furthermore, those moves which fail to give graded Morita equivalent Leavitt path algebras do not preserve the associated talented monoids. We also study the notion of the period of a vertex, and that of a graph, in relation to the talented monoid and the structure of Leavitt path algebras. 

The period of a vertex is the greatest common divisor of the lengths of all closed paths based at that vertex. For a strongly connected graph, all vertices have the same period, which is called the period of the graph. In particular, a graph is called aperiodic if its period is $1$.  The notion of period of a graph appears in the theory of Markov chains, symbolic dynamics and automata theory. As an example, a shift of finite type associated to an aperiodic graph is a mixing shift space (see~\cite[\S 4.5]{MR1369092}).
In the setting of graph $C^*$-algebras, Pask and Rho consider the period of the graph in \cite{MR2018239}. Using this notion, they characterise a graph $E$ for which the fixed point algebra under the gauge action of $S^1$, i.e., $C^*(E)^\gamma$, is a simple ring.
We show in Section~6 that the period of a graph can be described completely via its associated talented monoid.  More accurately, we show in Theorem~\ref{thm:stronglyconnected} that a finite graph $E$ with no sources is strongly connected of period $d$ if, and only if, \[T_E=\bigoplus_{i=0}^{d-1}\, {}^i I,\] where $I$ is a simple order ideal with ${}^d I=I$ (i.e, there exists a simple order ideal of period $d$). 

Using our results we give a fine description of purely infinite simple unital Leavitt path algebras. Namely, we show that if $L_{\mathsf k}(E)$ is purely infinite simple, then the ring of the zero component $L_{\mathsf k}(E)_0$ can be written as a direct sum of $d$ minimal two sided ideals, where $d$ is the period of the graph $E$ associated to this algebra (Theorem~\ref{hfgftrgfggf}). As an example, the following two graphs produce isomorphic purely infinite simple Leavitt path algebras, however the period of the graph $E$ is $1$ whereas the period of the graph $F$ is $2$. 
\medskip 

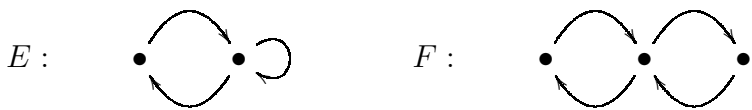
\[\xymatrix{
E:& \bullet \ar@/^1.5pc/[r] & \bullet \ar@/^1.5pc/[l]  \ar@(ru,rd) &  &
F:& \bullet \ar@/^1.5pc/[r]  & \bullet \ar@/^1.5pc/[r]  \ar@/^1.5pc/[l]& \bullet \ar@/^1.5pc/[l] 
}\]

\bigskip 
\noindent One can check that $L_{\mathsf k}(E)_0$ is a simple ring whereas $L_{\mathsf k}(F)_0\cong I\oplus J$, where $I$ and $J$ are minimal two sided ideals. Using the talented monoid, we show that although isomorphisms between Leavitt path algebras do not necessarily preserve the periods of the graphs, the graded isomorphisms do, which is another evidence that the talented monoid can be a complete invariant for graded Morita equivalence of graph algebras.

%We also put this program of classification in the broader framework of classification of certain ample groupoids via their \emph{graded type semigroup}, by proving that the talented monoid $T_E$ is the type semigroup of the skew product of the $\mathbb Z$-graded graph groupoid $\mG_E$ by $\mathbb Z$ (see~\S\ref{typesemigroup}). 

The paper is organised as follows: after this introduction we include a section of preliminaries, where we recall the relevant concepts that will be needed through the paper. In Section~\ref{typesemigroup} we show that the talented monoid of a graph can be obtained as the type semigroup of the skew product of the graph groupoid with $\mathbb{Z}$, and therefore we connect the graded classification conjecture with the program of classification of Steinberg algebras associated to Deaconu-Renault groupoids (via their graded type semigroup). Since Morita equivalence of Leavitt path algebras is preserved under graph moves (for a large class of graphs), we study the effect of these moves on the talented monoid in Section~\ref{sec:moveit}. Proceeding, in Section~\ref{sec:extremecycles} we describe extreme cycles in a graph in terms of the talented monoid and, in Section~\ref{sec:perioifhf}, we use the talented monoid to describe the period of a strongly connected graph. Furthermore, in Section~\ref{sec:perioifhf}, we describe the ideal generated by the ``primary colours'' of a graph, and give a finer description of the class of unital, purely infinite, simple Leavitt path algebras. 

%We finish this note in Section~\ref{pohfggf}, where we describe paradoxicality of the talented monoid in terms of a combinatorial property of the underlying graph. 

\section{Preliminaries}\label{premini}

In this section we briefly recall concepts and establish the notation which will be used throughout the paper. We refer the reader to \cite{MR3729290,MR2135030} for the theory of graph algebras, \cite{MR3700423} for monoids, and  \cite{MR584266,MR3299719} for topological groupoids.  In this work we will consider that $\mathbb{N}=\{0,1,2,\ldots\}$.

\subsection{Graphs}

A \emph{directed graph}  is a tuple $E=(E^0,E^1,s,r)$, where $E^0$ is a set of \emph{vertices}, $E^1$ a set of \emph{edges}, and $s,r\colon E^1\to E^0$ are functions, called the \emph{source} and \emph{range} maps. A graph $E$ is said to be \emph{row-finite} if for each vertex $u\in E^{0}$,
there are at most finitely many edges in $s^{-1}(u)$.  A vertex $u$ for which $s^{-1}(u)$ is infinite is called an \emph{infinite emitter}, whereas $u$ is called a \emph{sink} if $s^{-1}(u)$
is empty, and is said to be a \emph{source} if $r^{-1}(u)$ is empty.  If $u\in E^{0}$ is not a sink, nor an infinite emitter, then it is called a \emph{regular vertex}. We confine ourselves to row-finite graphs, as the original graded classification conjecture is for finite graphs,  although we expect that the results of the paper can be extended to arbitrary graphs, i.e., graphs with infinite emitters.

A \emph{path} $\mu$ in $E$ is a nonempty sequence $\mu=e_1e_2\cdots$ (finite or infinite) of edges such that $s(e_{i+1})=r(e_i)$ for all $i$. The length of a path is the number of terms in the sequence, and is denoted as $|\mu|$. The source map extends to paths as $s(e_1e_2\cdots )=s(e_1)$, and the range map also extends to finite paths as $r(e_1\cdots e_n)=r(e_n)$. Every vertex of $E$ is regarded as a path of length $0$, with itself as both its source and range.

A path $\mathfrak{c}=e_1e_2\cdots e_n$ is called a \emph{closed path based at $v$} if $v=s(\mathfrak{c})=r(\mathfrak{c})$. A \emph{cycle} in $E$ is a closed path $\mathfrak{c}=e_1e_2\cdots e_n$ such that $s(e_i)\not = s(e_j)$ for all $i\not = j$.  An \emph{exit} of a cycle $\mathfrak{c}=e_1\cdots e_n$ consists of an edge $f$ such that $s(f)=s(e_i)$ for some $i$ but $f\neq e_i$. The vertices $s(e_1),\ldots,s(e_n)$ are called the \emph{vertices of $\mathfrak{c}$}, and the set of these vertices is denoted by $\mathfrak{c}^0$, that is, $\mathfrak{c}^0=\{s(e_1),\ldots,s(e_n)\}$.  If $f$ is an exit of the cycle $\mathfrak{c}$, then a \emph{return} of $f$ to $\mathfrak{c}$ is a path $\mu$ such that $s(\mu)=r(f)$ and $r(\mu)\in \mathfrak{c}^0$. We say that a vertex $v$ \emph{connects} to a cycle $\mathfrak{c}$ if there exists a path $\mu$ with $s(\mu)=v$ and $r(\mu)\in \mathfrak{c}^0$.

We say that the graph $E$ is \emph{strongly connected} if for any two vertices $u$ and $v$ of $E$, there exist finite paths $\mathfrak{c}$ and $\mathfrak{d}$ in $E$ such that $s(\mathfrak{c})=r(\mathfrak{d})=u$ and $r(\mathfrak{c})=s(\mathfrak{d})=v$. 

For a vertex $v$ of a finite graph, the \emph{period} of $v$ is defined as the greatest common divisor of the lengths of all closed paths based on $v$. If $v$ is not contained in any cycle we set the period of $v$ to be zero. It is known that for a finite strongly connected graph $E$, all vertices have the same period which is defined to be the period of the graph and denoted by $d(E)$ (see \cite{MR2018239} and \cite[\S 4.5]{MR1369092}).

We say that a vertex $v$ \emph{flows} to the vertex $w$, or that $w$ is \emph{flowed} into from $v$, if $v=w$ or if there is a path from $v$ to $w$. A vertex $v$ in a graph $E$ \emph{has a bifurcation} if $|s^{-1}(v)|\geq 2$. A vertex $v$ is a \emph{line point} if there are no bifurcations nor cycles at any vertex $w$ which is flowed into from $v$.

We will distinguish several types of cycles. The aim is to characterise them in terms of the talented monoid $T_E$ associated to the graph $E$. 

If a cycle $\mathfrak{c}$ does/does not have an exit, then we say $\mathfrak{c}$ is \emph{cycle with/without exit}. An \emph{extreme cycle} is a cycle which admits an exit, and such that every finite path which exits from it admits a return to it. We say that a cycle is a \emph{cycle with no return exit} if the cycle has an exit, however no exit returns to the cycle.

More formally, an \emph{extreme cycle} in $E$ is a cycle $\mathfrak{c}=e_1e_2\cdots e_n$ on $E$ such that:
\begin{enumerate}[label=(\roman*)]
    \item $c$ has at least one exit;
    \item for every finite path $\lambda$ with $s(\lambda)\in \mathfrak{c}^0$, there exists another path $\mu$ such that $s(\mu)=r(\lambda)$ and $r(\mu)\in \mathfrak{c}^0$.
\end{enumerate}

A \emph{cycle with no return exit} in $E$ is a cycle $\mathfrak{c}=e_1e_2\cdots e_n$ on $E$ such that:
\begin{enumerate}[label=(\roman*)]
    \item $\mathfrak{c}$ has at least one exit;
    \item for every path $\lambda$ with $s(\lambda)\in \mathfrak{c}^0$ and $r(\lambda)\notin \mathfrak{c}^0$, there is no path $\mu$ such that $s(\mu)=r(\lambda)$ and $r(\mu)\in\mathfrak{c}^0$.
\end{enumerate}

The graph $E$ satisfies \emph{Condition (L)} if every cycle in $E$ has an exit. This means that every cycle $\mathfrak{c}$ has a vertex $v$ with $|s^{-1}(v)|\geq 2$.

\begin{example}

Consider the following graphs: 

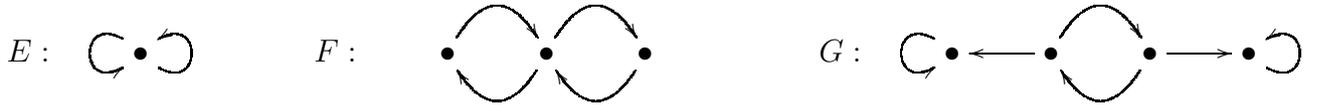
\[\xymatrix{
E:& \bullet \ar@(rd,ru) \ar@(lu,ld) & &
F:& \bullet \ar@/^1.5pc/[r] & \bullet \ar@/^1.5pc/[r] \ar@/^1.5pc/[l]& \bullet \ar@/^1.5pc/[l] & & 
G:& \bullet \ar@(lu,ld) & \bullet \ar[l] \ar@/^1.5pc/[r] & \bullet \ar@/^1.5pc/[l] \ar[r]& \bullet \ar@(rd,ru)
}\]

\medskip 
The graphs $E$ and $F$ are strongly connected with periods (see Section~\ref{sec:perioifhf}) $1$ and $2$, respectively. We will show that although their associated Leavitt path algebras are isomorphic, they are not graded isomorphic. Notice that the graph $G$ has two cycles without exits and a cycle with no return exit. 

On the opposite spectrum, the vertex $v$ in the following graph is a line point.

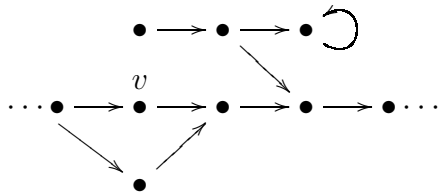
\[\xymatrix@R=.8pc @C=1.5pc{
    &\bullet \ar[r]
        &\bullet \ar[r] \ar[rd]
            &\bullet \ar@(rd,ru)\\
\cdots\bullet \ar[r] \ar[rd]
    &{\overset{\textstyle v\mathstrut}{\bullet}} \ar[r]
        &\bullet\ar[r]
            &\bullet\ar[r]
                &\bullet\cdots\\
&\bullet \ar[ru]
        &
            &
}\]

\end{example}

For row-finite graphs $E$ and $F$, a \emph{graph morphism} $f\colon E\rightarrow F$ consists of maps $f^0\colon E^0\rightarrow F^0$ and $f^1\colon E^1\rightarrow F^1$, such that, $s(f^1(e))=f^0(s(e))$ and 
$r(f^1(e))=f^0(r(e))$, for any edge $e\in E^1$. Furthermore, a morphism is \emph{complete} if $f^0$ is injective and 
$|s^{-1}(v)| = |s^{-1}(f^0(v))|$ if $v\in E$ is not a sink.

\subsection{Leavitt path algebras}
To a directed graph, one can associate an algebra generated by vertices and edges, subject to relations that ``locally'' on each vertex resemble those that were considered by William Leavitt in his seminal papers in the 1960's (see~\cite{MR3729290} for a comprehensive history). Such algebras, when associated to strongly connected graphs which are not a single cycle, are purely infinite simple, that is, each one-sided ideal contains an infinite idempotent.

\begin{definition}\label{llkas} 
For a row-finite graph $E$ and a unital ring $R$, we define the \emph{Leavitt path algebra of $E$}, denoted by $L_R(E)$, to be the algebra generated by the sets $\{v \mid v \in E^0\}$, $\{ \alpha \mid \alpha \in E^1 \}$ and $\{ \alpha^* \mid \alpha \in E^1 \}$ with the coefficients in $R$, subject to the relations 

\begin{enumerate}
\item $v_iv_j=\delta_{ij}v_i \textrm{ for every } v_i,v_j \in E^0$;

\item $s(\alpha)\alpha=\alpha r(\alpha)=\alpha \textrm{ and }
r(\alpha)\alpha^*=\alpha^*s(\alpha)=\alpha^* \textrm{ for all } \alpha \in E^1$;

\item $\alpha^* \alpha'=\delta_{\alpha \alpha'}r(\alpha)$, for all $\alpha, \alpha' \in E^1$;

\item $\sum_{\{\alpha \in E^1, s( \alpha)=v\}} \alpha \alpha^*=v$ for every $v\in E^0$ for which $s^{-1}(v)$ is nonempty.

\end{enumerate}
\end{definition}
In this note we only work with Leavitt path algebras with coefficients in a field ${\mathsf k}$. The elements $\alpha^*$ for $\alpha \in E^1$ are called \emph{ghost edges}. \index{ghost edge} One can show that $L_{\mathsf k}(E)$ is a ring with identity if and only if the graph $E$ is finite (otherwise, $L_{\mathsf k}(E)$ is a ring with local identities).

Setting $\deg(v)=0$, for $v\in E^0$, $\deg(\alpha)=1$ and $\deg(\alpha^*)=-1$ for $\alpha \in E^1$, we obtain a natural $\mathbb Z$-grading on the free ${\mathsf k}$-ring generated by $\big \{v,\alpha, \alpha^* \mid v \in E^0,\alpha \in E^1\big \}$. Since the relations in Definition~\ref{llkas} are all homogeneous, the ideal generated by these relations is homogeneous and thus we have a natural $\mathbb Z$-grading on $L_{\mathsf k}(E)$. The zero homogeneous component $L_{\mathsf k}(E)_0$ is an ultramatricial algebra (see the proof of \cite[Theorem 5.3]{MR2310414}) and thus if $L_{\mathsf k}(E)_0$ is unital it is a unit-regular ring (i.e., every \(x\in L_{\mathsf k}(E)_0\) may be written as \(x=xux\) for some unit \(u\) in \(L_{\mathsf k}(E)_0\)).

Among the attractions of the theory of Leavitt path algebras is that one can describe certain ring properties of these algebras based purely on the combinatorial properties of the associated graphs. We recall here one of these facts that we will later revisit \cite[p. 205]{MR2785945}.

\begin{theorem}\label{drumoyne}
Let $E$ be a finite graph and $\mathsf{k}$ be a field. The following are equivalent:
\begin{enumerate}
    \item $L_{\mathsf{k}}(E)$ is purely infinite and simple;

    \item The graph $E$ satisfies condition (L), has a cycle, and every vertex connects to every cycle.
\end{enumerate}
\end{theorem}

Using the talented monoids, we will add more details to this characterisation by taking into account the period of the graph as well (Theorem~\ref{hfgftrgfggf}). 

\subsection{Monoids}\label{balconyfreshwater}

Given a group $\Gamma$, a \emph{$\Gamma$-monoid} consists of a monoid $M$ equipped with an action of $\Gamma$ on $M$ (by monoid automorphisms). We denote the action of $\alpha\in\Gamma$ on $m\in M$ by ${}^\alpha m$. A monoid homomorphism $\phi\colon M_1 \rightarrow M_2$ between two $\Gamma$-monoids is called $\Gamma$-\emph{monoid homomorphism} if $\phi$ respects the actions of $\Gamma$, i.e., $
\phi({}^\alpha a)={}^\alpha \phi(a)$ for all $a\in M_1$. In this note we are concerned with commutative monoids. Every commutative monoid $M$ is equipped with a natural preordering: $y \leq x$ if $y+z=x$ for some $z\in M$. If $M$ is a $\Gamma$-monoid, this ordering is respected by the action of $\Gamma$. We say $M$ is \emph{conical} if $x+y=0$ implies that $x=y=0$, where $x,y \in M$. We say $M$ is \emph{cancellative} if $x_1+y=x_2+y$ implies $x_1=x_2$.

For a $\Gamma$-monoid $M$ we distinguish two types of submonoids. An \emph{order ideal} of $M$ is a submonoid $I$ which is also an ideal with respect to the natural order of $M$, i.e., if $x\in I$ and $y \leq x$, then $y\in I$. A \emph{$\Gamma$-order ideal} is an order ideal of $M$ which is closed under the action of $\Gamma$. Every $\Gamma$-order ideal $I$ of $M$ is a $\Gamma$-monoid on its own right, and the restriction of the natural order of $M$ to $I$ is the natural order of $I$.

Given $x\in M$, we denote by $[x]$ the order ideal generated by $x$, and by $\langle x\rangle$ the $\Gamma$-order ideal generated by $x$. We have 
\begin{align}\label{idealorderandnotorder}
[x]&=\{ y\in M \mid y \leq n x, \text{ for some } n \in \mathbb N \},\\ 
\langle x\rangle &= \{y\in M \mid y\leq\sum_{\alpha\in\Gamma} k_\alpha{}^\alpha x, \text{for some } k_\alpha\in\mathbb{N} \}. \notag
\end{align}
It is easy to see that for $\alpha \in \Gamma$, we have ${}^\alpha [x]=[{}^\alpha x]$. 

We adapt the notion of \emph{essential ideal} in algebra to the setting of (ordered) $\Gamma$-monoids.

\begin{definition}\label{essentialll}
    Let $\Gamma$ be a group and $M$ a $\Gamma$-monoid. A $\Gamma$-order ideal $I$ of $M$ is \emph{essential} if $I$ has nonzero intersection with every other nonzero $\Gamma$-order ideal of $M$.
\end{definition}

\subsection{Graph monoid and the talented monoid}

In this section we define the graph monoids that are the main interests of this paper. 

Given a row-finite graph $E$, we denote by $F_E$ the free commutative monoid generated by $E^0$.

\begin{definition}\label{def:graphmonoid}
    Let $E$ be a row-finite graph. The \emph{graph monoid} of $E$, denoted $M_E$, is the commutative 
monoid generated by $\{v \mid v\in E^0\}$, subject to
\[v=\sum_{e\in s^{-1}(v)}r(e)\]
for every $v\in E^0$ that is not a sink.
\end{definition}

The relations defining $M_E$ can be described more concretely as follows: First, define a relation $\to_1$ on $F_E$ as follows: for $\sum_{i=1}^n v_i  \in F$, and a regular vertex $v_j\in E^0$, where $1\leq j \leq n$,  
\[\sum_{i=1}^n v_i \to_1 \sum_{i\not = j }^n v_i+  \sum_{e\in s^{-1}(v_j)}r(e).\]
Then $M_E$ is the quotient of $F_E$ by the congruence generated by $\to_1$.

Let $\to$ be the smallest reflexive, transitive and additive relation on $F_E$ which contains (is coarser than) $\to_1$. Note that $\to$ is not symmetric, so it is not a congruence.

The relation $\to$ may be regarded as follows: If \(x=\sum_i x_i\) is an element of \(F_E\), we may ``let a vertex \(x_i\) flow'' to construct the element \(y_1=\left(\sum_{j\neq i}x_j\right)+\sum_{e\in s^{-1}(x_i)}r(e)\) with \(x\to y_1\). Repeating this procedure and ``letting a vertex of \(y_1\) flow'', we construct another element \(y_2\in F_E\) such that \(y_1\to y_2\). In other words, we simply apply the definition of \(\to_1\) to vertices in the representation of elements of \(F_E\). By the definition of \(\to\), every element \(y\in F_E\) such that \(x\to y\) may in fact be constructed from \(x\) by ``letting its vertices flow successively'' in this manner. The following proposition thus becomes clear:

\begin{proposition}\label{prop:flowpath}
Suppose that $x=\sum_i x_i$ and $y=\sum_j y_j$ are elements of $F_E$, where $x_i,y_j\in E^0$. If $x\to y$, then
\begin{enumerate}[label=(\arabic*)]
    \item For every $i$, there exists $j$ and a path from $x_i$ to $y_j$;

    \item For every $j$, there exists $i$ and a path from $x_i$ to $y_j$.
\end{enumerate}
\end{proposition}

 By the proposition above, a vertex $v$ flows to the vertex $u$ if, and only if, either $v=u$, or there exists $x\in F_E$ such that $u$ belongs to the decomposition of $x$ in vertices, and such that $v\to x$.

The following lemma is essential to the remainder of this paper, as it allows us to translate the relations in the definition of $M_E$ in terms of the simpler relation $\to$ in $F_E$.

\begin{lemma}[{\cite[Lemmas 4.2 and 4.3]{MR2310414}}]\label{confuu}
    Let $E$ be a row-finite graph.
    \begin{enumerate}[label=(\alph*)]
        \item (The Confluence Lemma) If $a,b\in F_E\setminus\left\{0\right\}$, then $a=b$ in $M_E$ if and only if there exists $c\in F_E$ such that $a\to c$ and $b\to c$. (Note that, in this case, $a=b=c$ in $M_E$.)

        \item If $a=a_1+a_2$ and $a\to b$ in $F_E$, then there exist $b_1,b_2\in F_E$ such that $b=b_1+b_2$, $a_1\to b_1$ and $a_2\to b_2$.
    \end{enumerate}
\end{lemma}

Now we define the \emph{talented monoid} $T_E$ of $E$, which encodes the graded structure of a Leavitt path algebra $L_{\mathsf k}(E)$ as well.

\begin{definition}\label{talentedmon}
Let $E$ be a row-finite directed graph. The \emph{talented monoid} of $E$, denoted $T_E$, is the commutative 
monoid generated by $\{v(i) \mid v\in E^0, i\in \mathbb Z\}$, subject to
\[v(i)=\sum_{e\in s^{-1}(v)}r(e)(i+1)\]
for every $i \in \mathbb Z$ and every $v\in E^{0}$ that is not a sink. The additive group $\mathbb{Z}$ of integers acts on $T_E$ via monoid automorphisms by shifting indices: For each $n,i\in\mathbb{Z}$ and $v\in E^0$, define ${}^n v(i)=v(i+n)$, which extends to an action of $\mathbb{Z}$ on $T_E$. Throughout we will denote elements $v(0)$ in $T_E$ by $v$.
\end{definition}

The crucial ingredient for us is the action of $\mathbb Z$ on the monoid $T_E$. The general idea is that the monoid structure of $T_E$ along with the action of $\mathbb Z$ resemble the graded ring structure of a Leavitt path algebra $L_{\mathsf k}(E)$.

The talented monoid of a graph can also be seen as a special case of a graph monoid, which we now describe. The \emph{covering graph} of $E$ is the graph $\overline{E}$ with vertex set $\overline{E}^0=E^0\times\mathbb{Z}$, and edge set $\overline{E}^1=E^1\times\mathbb{Z}$. The range and source maps are given as
\[s(e,i)=(s(e),i),\qquad r(e,i)=(r(e),i+1).\] Note that the graph monoid $M_{\overline E}$ has a natural $\mathbb Z$-action by ${}^n (v,i)= (v,i+n)$. The following theorem allows us to use the Confluence Lemma~\ref{confuu} for the talented monoid $T_E$ by identifying it with $M_{\overline E}$. 

\begin{theorem}[{\cite[Lemma 3.2]{MR4040730}}]\label{hgfgfgggf}
    The correspondence
    \begin{align*}
        T_{E} &\longrightarrow M_{\overline{E}}\\
        v(i) &\longmapsto (v,i)
    \end{align*}
    induces a $\mathbb Z$-monoid isomorphism.
\end{theorem}

Note that $M_E$ is the quotient of $T_E$ obtained by identifying elements of $T_E$ which belong to the same $\mathbb{Z}$-orbit. The respective quotient map,
\begin{align}\label{sungftgdtd}
T_E&\longrightarrow M_E\\
v(i)&\longmapsto v,\notag
\end{align}
is also called the \emph{forgetful} homomorphism. It follows that any $\mathbb{Z}$-monoid homomorphism between $T_E$ and $T_F$, for row-finite graphs $E$ and $F$, induces a monoid homomorphism between $M_E$ and $M_F$.

By \cite[Proposition 5.7]{MR3781435}, $T_E$ is $\mathbb{Z}$-monoid isomorphic to the monoid $\mathcal{V}^{\gr}(L_{\mathsf{k}}(E))$ of isomorphism classes of graded finitely generated projective $L_{\mathsf{k}}(E)$-modules, where $\mathsf{k}$ is an arbitrary field. It follows that $T_E$ is conical (the same is true for $M_E$). By \cite[Corollary 5.8]{MR3781435}, $T_E$ is also cancellative. These two facts may also be proved directly using the Confluence Lemma~\ref{confuu}.

We note that if $\phi\colon E\rightarrow F$ is a complete graph homomorphism, then $\phi$ extends to a natural $\mathbb Z$-monoid homomorphism $\overline{\phi}\colon T_E \rightarrow T_F$. In the case of Leavitt path algebras, the map $\phi$ induces an injective ring homomorphism $\overline{\phi}: L_{\mathsf k}(E) \rightarrow L_{\mathsf k}(F)$. However injectivity does not follow  in the setting of talented monoids, as the following example shows. For the graphs $E$ and $F$, 
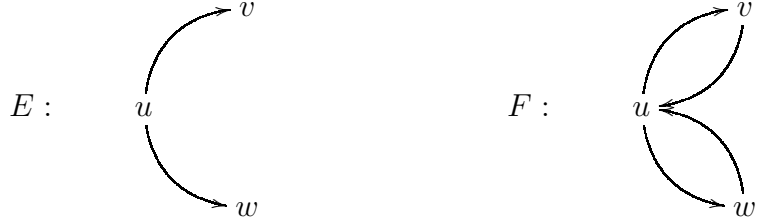
 \[\xymatrix{
  &  & v   &&&&&  v \ar@/^1pc/[dl]  \\
E:& u  \ar@/^1pc/[ur] \ar@/_1pc/[dr] &  &&& F: & u  \ar@/^1pc/[ur] \ar@/_1pc/[dr]\\
 & & w &&&&& w  \ar@/_1pc/[ul]
}\]
the $\mathbb Z$-monoid homomorphism $\overline{\phi}\colon T_E \rightarrow T_F$ is not injective, as in $T_E$ we have $u \not = u(2)+u(2)$, whereas their images under $\overline{\phi}$ coincide.

Next we describe the talented monoid for a couple of graphs which will play a role later (see Example~\ref{example1}).

\begin{example}\label{exampleR}
Consider the following graphs: 

\[\xymatrix{
E:& \bullet_v \ar@(rd,ru) \ar@(lu,ld) & &
F:& \bullet \ar@/^1.5pc/[r] & \bullet \ar@/^1.5pc/[r] \ar@/^1.5pc/[l]& \bullet \ar@/^1.5pc/[l] 
}\]

\bigskip

Let $\mathbb{N}[1/2] :=\Big \{\dfrac{m}{2^j}:m \in \mathbb{N}, j\in \mathbb{Z} \Big\}$ be a monoid equipped with the action of $\mathbb Z$ as follows ${}^i \big (\dfrac{m}{2^j}\big ) = \dfrac{m}{2^{i+j}}$, where $i\in \mathbb Z$.  Notice that in $T_E$ we have that $v(i)=2v(i+1)$ for all $i\in \mathbb{Z}$. Then \(T_E\) and \(\mathbb{N}[1/2]\) are $\mathbb Z$-isomorphic via a map taking \(v(i)\) to \(\dfrac{1}{2^i}\). So 
\[T_E \cong \mathbb{N}[1/2], \text{ with } {}^1 a = \frac{1}{2} a.
\]
Similarly,
\[T_F \cong \mathbb{N}[1/2]\oplus \mathbb{N}[1/2] , \text{ with } {}^1 (a,b)= (\frac{1}{2}b,a),
\]
where the isomorphism is given by identifying \((1,0)\) with the middle vertex of \(F\) and \((0,1)\) with any lateral vertex.
\end{example}

\subsection{Groupoids}\label{gftgftrgtd}
We follow the language of groupoids described in \cite[Section~3.3]{ClarkHazrat} and only recall here a few essential notations/concepts. We need this notions to interpret the talented monoids as a type semigroup of graph groupoids.

If $\Gamma$ is a group, then $\mG$ is called a $\Gamma$-\emph{graded groupoid} if there is a functor $c\colon \mG \rightarrow \Gamma$. For $\gamma \in \Gamma$ we set $\mG_\gamma:=c^{-1}(\gamma)$.
In the topological setting, we call a groupoid $\mG$ a $\G$-graded groupoid if
the function $c  \colon \mG \to \G$ is continuous with respect to the discrete topology on $\Gamma$; such a function $c$ is called a \emph{cocycle} on $\mG$. A compact open bisection $U\subseteq\mG$ is \emph{graded} if $U\subseteq\mG_\gamma$ for some $\gamma\in\Gamma$. Let $\mG$ be a $\Gamma$-graded ample Hausdorff groupoid. Set
\begin{align*}
\mG^{a}&=\big \{U\mid U \text{~is a compact open bisection of~} \mG \big \},\\
\mG^{h}&=\big \{U\mid U \text{~is a graded compact open bisection of~} \mG \big \}.
\end{align*}

Then $\mG^{a}$ and $\mG^{h}$ are inverse semigroups under the multiplication $U\cdot V=UV$ and inner inverse $U^*=U^{-1}$. Furthermore, the map $c \colon\mG^{h} \backslash \{\varnothing\} \rightarrow \Gamma, U\mapsto \gamma$, if $U\subseteq \mG_\gamma$, makes $\mG^{h}$ a graded inverse semigroup with $\mG^h_{\gamma}=c^{-1}(\gamma)$, $\gamma \in \Gamma$, as the graded components. If from the outset we consider $\mG$ as a trivially graded groupoid (i.e., $\Gamma=\{1\}$), then $\mG^h=\mG^a$.

Given a commutative ring $R$ with identity, the \emph{Steinberg $R$-algebra} associated to an ample groupoid $\mG$, and denoted by $A_R(\mG)$, is the contracted semigroup algebra $R\mG^h$, modulo the ideal generated by $B + D - B \cup D$, where $B,D,B\cup D \in \mG^h_\g$, $\g \in \Gamma$ and $B\cap D = \varnothing$ (\cite[Theorem 3.10]{MR3274831}, \cite[Definition~3.2]{ClarkHazrat}). This is the algebraic counterpart of the groupoid $C^*$-algebras systematically studied by Renault~\cite{MR3299719}.

Returning to the graph context, let $E=(E^{0}, E^{1}, r, s)$ be a row-finite graph. The graph groupoid associated with $E$ can be defined in terms of a partially defined shift map on a space and the Renault-Deaconu construction (see \cite{BCW, Deaconu95,Renault00}). We recall the space below, so that we can recall the graph groupoid. We denote by $E^{\infty}$ the set of
infinite paths in $E$ and by $E^{*}$ the set of finite paths in $E$. Set
\[
X := E^{\infty}\cup \{\mu\in E^{*} \mid r(\mu) \text{ is a sink}\}.
\]
For $\mu\in E^{*}$ define
\[
Z(\mu)= \{\mu x \mid x \in X, r(\mu)=s(x)\}\subseteq X.
\]
The sets $Z(\mu)$ constitute a basis of compact open sets for a locally
compact Hausdorff topology on $X=\mG_{E}^{(0)}$. The \emph{graph groupoid} associated with $E$ is the groupoid
\[
\mG_{E} := \big \{(\a x,|\a|-|\b|, \b x) \mid \a, \b\in E^{*}, x\in X, r(\a)=r(\b)=s(x) \big \}.
\]
We view each $(x, k, y) \in \mG_{E}$ as a morphism with range $x$ and source $y$. The
formulas
\[(x,k,y)(y,l,z)= (x,k + l,z)\quad\text{ and }(x,k,y)^{-1}= (y,-k,x)\]
define composition
and inverse maps on $\mG_{E}$ making it a groupoid with unit space $ \mG_{E}^{(0)}=\{(x, 0, x) \mid x\in X\}$, which we will identify with the set $X$. 
The map $c\colon \mG_E \rightarrow \mathbb Z; (x,l,y)\mapsto l$ makes this groupoid a $\mathbb Z$-graded groupoid. The Steinberg algebra of this groupoid coincides with the Leavitt path algebra associated to the graph \cite[Example 3.2]{MR3299719}.

\section{Talented monoid, type semigroup and the classification of graph algebras}\label{typesemigroup}

In this short section we recall %%%defined above
%the definition of the talented monoid $T_E$ associated to a row-finite graph $E$ and
the graded classification conjecture related to the talented monoid. It is conjectured that, for a row-finite graph $E$, the talented monoid of $E$ along with its $\mathbb{Z}$-action is a complete graded Morita equivalence invariant for Leavitt and graph $C^*$-algebras.

We first prove that $T_E$ can be obtained as the type semigroup of the skew product of the graph groupoid with $\mathbb{Z}$. This allows us to put our classification conjecture in a larger framework of classifying certain ample groupoid algebras via their type semigroups. 
For this we need to recall the type semigroup (or type monoid) of an inverse semigroup. 

Let $S$ be an inverse semigroup with $0$ and denote by $E(S)$ the semilattice of idempotents of $S$. We say that $x,y\in S$ are \emph{orthogonal}, written $x \perp y$, if $x^* y =y x^*=0$. A \emph{Boolean inverse semigroup} is an inverse semigroup $S$ such that $E(S)$ is a Boolean ring (a ring with $x^2 = x$ for all $x$), and every orthogonal pair $x, y \in S$ has a supremum, denoted $x \oplus y \in S$ (see \cite[Definition 3.1.6]{MR3700423} for the notion of Boolean inverse semigroups). (These semigroups are called \emph{weakly Boolean} in \cite{MR3077869}.)

\begin{definition}\label{def:typesemigroup}
Let $S$ be a Boolean inverse semigroup. The \emph{type semigroup} of $S$ is the commutative monoid $\Typ(S)$ generated by symbols $\typ(x)$, where $x \in E(S)$, subject to the relations
\begin{enumerate}
    \item $\typ(0) = 0$,
    
    \item $\typ(x) = \typ(y)$, whenever there is $s \in S$ such that $x=ss^*$ and $y=s^*s$,
    
    \item $\typ(x \oplus y) = \typ(x) + \typ(y)$,  whenever $x \perp y$.
    
\end{enumerate}
\end{definition}

One of the main examples of type semigroups for us are those coming from the compact open bisections of an ample groupoid, namely $\mG^a$ which is a Boolean inverse semigroup (see \S\ref{gftgftrgtd}). One then defines the \emph{type semigroup} of $\mG$, by $\Typ(\mG):=\Typ (\mG^a)$. If the groupoid $\mG$ is $\Gamma$-graded, then one can show that $\Typ(\mG^a) \cong \Typ(\mG^h)$. 

The majority of interesting groupoids come with a grading. Thus one can form the skew product of the groupoid with the grade group. The object of interest for us is the type semigroup coming from skew-product groupoids. We recall the notion of skew product groupoid below  (see \cite[Definition 1.6]{MR584266}).

\begin{definition} \label{defsp}
Let $\mG$ be an ample Hausdorff groupoid, $\G$ a discrete group and 
$c\colon \mG \to \Gamma$ a cocycle. The \emph{skew-product} of $\mG$ by $\Gamma$ is the groupoid $\mG \times_c \Gamma$ such that $(x,\alpha)$ and $(y, \beta)$ are composable if $x$ and $y$ are
composable and $\beta=\alpha c(x)$. The composition is then given by 
$\big (x, \alpha\big)\big(y,\alpha c(x) \big)=(xy, \alpha)$ with the inverse $(x, \alpha)^{-1}=(x^{-1},\alpha c(x))$.
\end{definition}

For a $\Gamma$-graded ample groupoid $\mG$, the skew-product $\mG\times_c\Gamma$ is also ample, where the topology is induced from the product topology on $\mathcal G \times \Gamma$. The unit space of $\mG\times_c\Gamma$ is $\mG^{(0)}\times\Gamma$. The idempotents of $(\mG\times_c\Gamma)^a$ are precisely the compact-open subsets of $\mG^{(0)}\times\Gamma$.   Since it is Hausdorff, these are the disjoint unions of sets of the form $U\times\alpha$, where $U$ is a compact-open subset of $\mG^{(0)}$ and $\alpha\in\Gamma$.

We now define the graded type semigroup of the $\Gamma$-graded ample groupoid $\mG$.

\begin{definition}\label{chuva}
Let $\mG$ be a $\Gamma$-graded ample groupoid. The \emph{graded type semigroup} of $\mG$ is defined as $\Typ^{\gr}(\mG):=\Typ(\mG\times_c \Gamma)$. Thus $\Typ^{\gr}(\mG)$ is generated by symbols $\typ(U \times \alpha)$, where $U$ is a compact open set of $\mG^{0}$ and $\alpha \in \Gamma$. There is an action of $\Gamma$ on $\Typ^{\gr}(\mG)$ defined 
 on generators by 
\[{}^\beta \typ(U \times \alpha) = \typ(U \times \beta \alpha)\]
and extended linearly to all elements. 
\end{definition}

It appears that this monoid along with the action of the group $\Gamma$ could encompass a substantial amount of information about the groupoid and its associated groupoid algebras. One of the most natural (and interesting) classes of étale groupoids are 
Deaconu-Renault groupoids, which are naturally $\mathbb Z$-graded. For reader's convenience, we recall the definition of a Deaconu-Renault groupoid below. 

\begin{definition}
Let $(X,\sigma)$ be a pair consisting of a locally compact Hausdorff space $X$, and a local homeomorphism $\sigma: \Dom(\sigma)\longrightarrow \Img(\sigma)$ from an open set $\Dom (\sigma)\subseteq X$ to an open set $\Img(\sigma)\subseteq X$. Inductively define $D_n:=\Dom(\sigma^n)=\sigma^{-1}(\sigma^{n-1})$. The \emph{Deaconu-Renault groupoid} associated with $(X,\sigma)$ is defined as \[G(X,\sigma)=\displaystyle\bigcup_{n,m\in \mathbb N} \Big \{(x,n-m,y)\in D_n\times \{n-m\}\times D_m \mid  \sigma^n(x)=\sigma^m(y)\Big \},\] 
equipped the topology with basic open sets \[Z(U,n,m,V):=\big \{(x,n-m,y):x\in U,y\in V,\text{ and }\sigma^n(x)=\sigma^m(y)\big\},\] 
indexed by quadruples $(U,n,m,V)$, where $n,m\in \mathbb N$, $U\subseteq D_n$ and $V\subseteq D_m$ are open and $\sigma^n|_U$ and $\sigma^m|_V$ are homeomorphism.

\end{definition}

It is thus plausible to consider the following line of enquiry.

\begin{prob}\label{genproblem}
Describe the class of Deaconu-Renault groupoids $\mG$ such that the graded type semigroup $\Typ^{\gr}(\mG)$, as a $\mathbb Z$-monoid, is a complete invariant for Steinberg and groupoid $C^*$-algebras. 
\end{prob}

Recall that for a directed graph $E$, its associated graph groupoid $\mG_E$ is a prototype of a Deaconu-Renault groupoid. Their Steinberg and groupoid $C^*$-algebras become Leavitt and graph $C^*$-algebras, respectively: $A_{\mathsf k}(\mG_E)\cong L_{\mathsf k}(E)$ \cite[Example 3.2]{MR3299719} and $C^*(\mG_E)\cong C^*(E)$ \cite[Proposition 4.1]{MR1432596}.

We will show that Problem~\ref{genproblem} for the graph groupoid is in fact the graded isomorphism conjecture posed in \cite{MR3004584}. Define the natural map, 
\begin{align*}
\phi\colon T_E &\longrightarrow \Typ^{\gr}(\mG_E)\\
v(i) &\longmapsto \typ\big(Z(v)\times i\big)
\end{align*}
on the generators and extend it to elements of $T_E$. Here we directly show how this map gives a well-defined homomorphism. In Lemma~\ref{lemtal}, using the machinery developed in \cite{arabosa}, we show that this map is indeed an isomorphism. 

We need to show that if  
\begin{equation*}
v(i)=\sum_{e\in s^{-1}(v)}r(e)(i+1), 
\end{equation*}
then 
\begin{equation}\label{thisis}
\typ\big(Z(v)\times i\big) = \sum_{e\in s^{-1}(v)} \typ \big (Z(r(e)) \times (i+1) \big ). 
\end{equation}
Suppose $p \in E^*$ is a finite path. By Definition~\ref{defsp} of the skew-product, we have 
\begin{align*}
\Big (Z\big(r(p),p\big)\times (i+|p|) \Big)\Big (Z\big(p,r(p)\big)\times i \Big)
    & = \Big (Z\big(r(p),r(p)\big)\times (i+|p|) \Big)\\
\Big (Z\big(p, r(p)\big)\times i \Big)\Big (Z\big(r(p),p\big)\times (i+|p|) \Big)
    & = \Big (Z\big(p,p\big)\times i \Big).
\end{align*}
Relation (2) in the Definition~\ref{def:typesemigroup} of type semigroup now gives
\begin{equation}\label{hgnhuiss}
\typ\big(Z\big(p,p\big)\times i\big)=\typ\big(Z\big(r(p),r(p)\big)\times (i+|p|)\big).
\end{equation}
In particular for $e\in E^1$, we get 
\begin{equation}\label{hgnhui}
\typ\big(Z\big(e,e\big)\times i\big)=\typ\big(Z\big(r(e),r(e)\big)\times (i+1)\big).
\end{equation}
Since
\[Z(v)\times i = \bigsqcup_{e\in s^{-1}(v)} Z(e,e)\times i, \] 
by relation (3) of Definition~\ref{def:typesemigroup} we have 
\[\typ\big(Z(v)\times i\big) = \sum_{e\in s^{-1}(v)} \typ \big (Z(e,e)\times i \big ).\] 
Replacing the right hand side by using equalities~\eqref{hgnhui}, we obtain Equation~\eqref{thisis}. This shows that $\phi$ is well-defined (and surjective). 
We use a recent result of Ara, Bosa, Pardo and Sims on the type semigroup of (separated) graphs~\cite{arabosa} to give a direct proof that this map is an isomorphism. 

\begin{lemma}\label{lemtal}
Let $E$ be a row-finite graph. Then there is a $\mathbb Z$-monoid isomorphism \begin{align*}
    T_E &\longrightarrow \Typ^{\gr}(\mG_E),\\
    v(i) & \longmapsto \typ\big (Z(v) \times i \big). 
\end{align*}
\end{lemma}
\begin{proof}
Consider the maps 

\begin{alignat*}{4}
 & T_E &&\stackrel{\phi_1}{\longrightarrow} M_{\overline E} &&\stackrel{\phi_2}\longrightarrow \Typ(\mG_{\overline E}) &&\stackrel{\phi_3}\longrightarrow \Typ^{\gr}(\mG_E)\\
&v(i) &&\longmapsto (v,i) && \longmapsto
\typ\big(Z(v,i)\big) && \longmapsto \typ\big(Z(v) \times i\big). 
\end{alignat*}
The map $\phi_1$ is the monoid isomorphism of Theorem~\ref{hgfgfgggf}. The isomorphism of $\phi_2$ follows from \cite[Theorem 7.5]{arabosa}. Since $\mG_{\overline E} \cong \mG \times \mathbb Z$ (see~\cite[Theorem 2.4]{MR1738948}), the isomorphism $\phi_3$ follows. We check that the composition of these maps, call it $\phi$, is a $\mathbb Z$-monoid isomorphism.  For $v\in E^0$ and $i,n\in \mathbb Z$ we have 
\[\phi({}^n v(i))=\phi(v(i+n))=\typ\big(Z(v)\times (i+n)\big)={}^n \typ\big(Z(v)\times i\big)={}^n \phi(v(i)),\]
and thus, by linearity, $\phi$ is a $\mathbb Z$-monoid isomorphism.
\end{proof}

Before we relate the graded type semigroups to the graded classification conjecture, we recall the notion of graded Morita equivalence. 

Let $A$ be a $\Gamma$-graded unital ring. Denote by $\Gr A$ the category of graded right $A$-modules and graded homomorphisms. For $\alpha \in \Gamma$, let $\mathcal T_{\alpha}: \Gr A \rightarrow \Gr A$ be the $\alpha$-shift auto-equivalence functor, i.e., $\mathcal T_{\alpha}(M)=M(\alpha)$ for any $A$-module $M$ and $\mathcal{T}_\alpha$ is the identity on morphisms.. We say the graded rings $A$ and $B$ are \emph{graded Morita equivalent} if there is an equivalence functor  $\phi:\Gr A \rightarrow \Gr B$ such that $\phi \mathcal T_{\alpha} = \mathcal T_{\alpha} \phi$, for any $\alpha \in \Gamma$.

In the setting of graph $C^*$-algebras, recall that if $E$ is a graph then there is a gauge action $\gamma^E:\mathbb T \rightarrow \Aut(C^*(E))$, given by $\gamma^E_z(p_v)=p_v$ and $\gamma^E_z(s_e)=zs_e$, for all $v\in 
E^0$ and $e\in E^1$. An isomorphism of $C^*$-algebras is \emph{graded} if it preserves the gauge action. In this case, for graphs $E$ and $F$, we write 
$\big (C^*(E), \gamma^E \big) 
\cong \big (C^*(F), \gamma^F \big) 
$. The graph $C^*$-algebras $C^*(E)$ and $C^*(F)$ are \emph{stably graded Morita equivalent} if 
\[\big (C^*(E)\otimes \mathcal K, \gamma^E\otimes \id \big) 
\cong \big (C^*(F)\otimes \mathcal K, \gamma^F\otimes \id \big).
\]

Combining Lemma~\ref{lemtal} with the fact that the talented monoid $T_E$ is the positive cone of the graded Grothendieck group $K_0^{\gr}(L_{\mathsf k}(E))$, the Problem~\ref{genproblem} on the level of graph groupoids reduces to the Graded Classification Conjecture (\cite{MR3729290,MR3319704,MR3004584}).

\begin{conj}\label{conjehfyhtr}
Let $E$ and $F$ be finite graphs, $T_E$ and $T_F$ the associated talented monoids and ${\mathsf k}$ a field. Then the following are equivalent.

\begin{enumerate}

\item There is a $\mathbb Z$-monoid isomorphism
$T_E \rightarrow T_F$; 

\item The $C^*$-algebras $C^*(E)$ and $C^*(F)$ are stably graded isomorphic. 

\item The Leavitt path algebras $L_{\mathsf k}(E)$ and $L_{\mathsf k}(F)$ are graded Morita equivalent. 

\end{enumerate}

\end{conj}

Furthermore if the $\mathbb Z$-monoid isomorphism $\phi\colon T_E \rightarrow T_F$ preserves the order-unit, i.e., 
\[\phi\left(\sum_{u\in E^0}u\right)=\sum_{u\in F^0} u,\] then the algebras should be (graded/gauge invariant) isomorphic (see~\cite{MR3004584} for the notion of an order unit in the graded setting).

\section{Graph moves}\label{sec:moveit}

 We start this section with a general question: If $E$ and $F$ are row-finite graphs such that the Leavitt path algebras $L_{\mathsf{k}}(E)$ and $L_{\mathsf{k}}(F)$ are equivalent in some sense (isomorphic, diagonally preserving isomorphic, graded isomorphic, Morita equivalent, etc.), how do the geometry of $E$ and $F$ relate? 

It turns out that in some cases this question has a very precise answer. Namely, given appropriate conditions on the graphs at hand, it can be shown that the Morita equivalence of Leavitt path algebras $L_{\mathsf{k}}(E)$ and $L_{\mathsf{k}}(F)$ implies that $E$ can be transformed into $F$ by means of some basic ``graph moves''.  For example, for simple Leavitt path algebras of finite graphs with no sinks, in \cite{MR2785945} the authors show that $K_0 (L_K(E)) \cong K_0(L_K(E) )$ and $\sgn (\det (I-A^t_E)) = \sgn (\det (I-A^t_F))$ implies that $L_K(E)$ is Morita equivalent to $L_K(F)$, and, moreover, in this case $E$ may be transformed into $F$ by a sequence of basic moves. Similar results, for non finite graphs, can be found in \cite[Theorem~8.12]{MR3341817} and \cite[Theorem~7.4]{MR3045151}. In the context of graph $C^*$-algebras, related results hold for Morita equivalence. See \cite{arxiv1611.07120} and \cite{MR3837603} for further references.

%*****************************************
\comment{

as described in the theorems below.

\begin{theorem}[{\cite[Theorem 8.12]{MR3341817}}]
Suppose that the field $\mathsf{k}$ is a number field and that $E$ and $F$ are graphs with $L_{\mathsf{k}}(E)$ and $L_{\mathsf{k}}(F)$ simple. Then the following are equivalent:
    \begin{enumerate}
        \item $L_{\mathsf{k}}(E)$ and $L_{\mathsf{k}}(F)$ are Morita equivalent.
        \item $K_0^{\mathrm{alg}}(L_{\mathsf{k}}(E))\cong K_0^{\mathrm{alg}}(L_{\mathsf{k}}(F))$ and $K_6^{\mathrm{alg}}(L_{\mathsf{k}}(E))\cong K_6^{\mathrm{alg}}(L_{\mathsf{k}}(F))$
        \item $E$ can be transformed into $F$ by finitely many of moves (S), (O), (I) and (R) and their inverses.
    \end{enumerate}
\end{theorem}

\begin{theorem}[{\cite[Theorem 7.4]{MR3045151}}]
    Suppose that $E$ and $F$ are graphs with a finite number of vertices and an infinite number of edges such that $L_{\mathsf{k}}(E)$ and $L_{\mathsf{k}}(F)$ simple. Then the following are equivalent:
    \begin{enumerate}
        \item $L_{\mathsf{k}}(E)$ and $L_{\mathsf{k}}(F)$ are Morita equivalent.
        \item $K_0^{\mathrm{alg}}(L_{\mathsf{k}}(E))\cong K_0^{\mathrm{alg}}(L_{\mathsf{k}}(F))$, and $E$ and $F$ have the same number of singular vertices.
        \item $E$ can be transformed into $F$ by finitely many of moves (S), (O), (I) and (R) and their inverses.
    \end{enumerate}
\end{theorem}

Similar results hold for Morita equivalence of graph $C^*$-algebras. See \cite{arxiv1611.07120} and \cite{MR3837603} for further references.
}
%*****************************************

In this section we will consider, among standard graph moves, those which yield graded Morita equivalent Leavitt path algebras, and prove that these moves also yield $\mathbb{Z}$-isomorphic talented monoids. This serves as further evidence to the claim that talented monoids are complete graded Morita equivalence invariant for Leavitt path algebras. In \S\ref{othermoves} we will discuss other standard moves that give Morita equivalent Leavitt path algebras, however would not give graded Morita equivalence.

\subsection*{Move (S): Source removal}

\begin{definition}
    Let $E$ be a row-finite graph and $v\in E^0$ a source which is also a regular vertex (i.e., not a sink). We say that $E_{\setminus v}$ -- the graph obtained by restricting $E$ to $E^0\setminus\left\{v\right\}$ -- is formed by performing \emph{Move (S)} on $E$.
\end{definition}

\begin{proposition}\label{valenjov}
Let $E$ be a row-finite graph. Let $v\in E^0$ be a source which is not a sink. Then $T_{E_{\backslash v}}$ is $\mathbb Z$-monoid isomorphic to $T_E$.
\end{proposition}
\begin{proof}
Since the natural map $E_{\backslash v} \rightarrow E; \ u \mapsto u$, is a complete graph morphism, it induces a $\mathbb{Z}$-monoid homomorphism $\phi\colon T_{E_{\backslash v}} \rightarrow T_E$. Writing $v=\sum_{e\in s^{-1}(v)}r(e)(1)$, since all vertices $r(e)\in T_{E_{\backslash v}}$, the map $\phi$ is surjective. On the other hand, if $\phi(x)=\phi(y)$, by the Confluence Lemma~\ref{confuu}, we have $x\rightarrow c$ and $y\rightarrow c$ in the graph $E$. Since the vertex $v$ does not appear in any presentation of $x$ and $y$, we thus have $x\rightarrow c$ and $y\rightarrow c$ in $E_{\backslash v}$ as well. This shows that $\phi$ is injective. 
\end{proof}

\subsection*{Move (I): In-splitting}

\begin{definition}[{\cite[Definition 6.3.20]{MR3729290}}]\label{def:insplit}
    Let $E$ be a directed graph. For each $v\in E^0$ with $r^{-1}(v)\neq 0$, take a partition $\left\{\mathscr{E}^v_1,\ldots,\mathscr{E}^v_{m(v)}\right\}$ of $r^{-1}(v)$. We form a new graph $F$ as follows:
    \[F^0=\left\{v_i \mid v\in E^0,1\leq i\leq m(v)\right\}\cup\left\{v \mid r^{-1}(v)=\varnothing\right\}\]
    \[F^1=\left\{e_j \mid e\in E^1,1\leq j\leq m(s(e))\right\}\cup\left\{e \mid r^{-1}(s(e))=\varnothing\right\},\]
    with source and range maps defined as follows: If $r^{-1}(s(e))\neq\varnothing$, choose $i$ such that $e\in\mathscr{E}^{r(e)}_i$, and set
    \[s(e_j)=s(e)_j,\qquad r(e_j)=r(e)_i, \text{ where } 1\leq j \leq m(s(e)).\]
    If $r^{-1}(s(e))=\varnothing$, set $s(e)$ as the original source of $e$, and $r(e)=r(e)_i$, where $i$ is chosen so that $e\in\mathscr{E}^{r(e)}_i$.
    
    The graph $F$ is called an \emph{in-split} of $E$, and conversely $E$ is called an \emph{in-amalgam} of $F$. We say that $F$ is formed by performing \emph{Move (I)} on $E$.
\end{definition}

If a graph $E$ has no sources nor sinks, and $F$ is a graph obtained from from $E$ by taking a series of in-splits and in-amalgam, then the associated Leavitt path algebras are graded Morita equivalent (\cite[Proposition~15]{MR3045160}; see also~\cite[Proposition 6.3.22]{MR3729290}). As talented monoids are conjectured to be complete invariants for the (graded) Morita equivalence, we prove here that they are preserved by in-splits and in-amalgams. This also shows how the talented monoid can capture the internal structures of the graphs, without going into the algebraic structures associated to the graphs. 

\begin{theorem}\label{inhfgft} Let $E$ be a row-finite graph. 
    If the graph $E$ does not have any sinks and $F$ is an in-split of $E$, then the map
    \[\phi\colon T_E\to T_F,\quad {}^kv\mapsto {}^kv_i,\]
    where $1\leq i\leq m(v)$ is chosen arbitrarily, is a $\mathbb{Z}$-monoid isomorphism.
\end{theorem}
\begin{proof}
    We will use the same notation as in Definition \ref{def:insplit}. To prove that $\phi$ is well-defined, it is sufficient to concentrate on the case $k=0$.
    
    First we prove that if $r^{-1}(v)\neq\varnothing$, then $v_i=v_j$ in $T_F$ for any $1\leq i,j\leq m(v)$. Let $i$ and $j$ be fixed. On one hand, $v_i$ is not a sink in $F$, so we have
    \[v_i=\sum\left\{{}^1r(e_k)\mid s(e_k)=v_i\right\}.\]
    Note that $s(e_k)=v_i$ if and only if $k=i$ and $s(e)=v$, that is,
    \[v_i=\sum\left\{{}^1r(e_i)\mid s(e)=v\right\},\]
    and similarly for $j$. Now note that $r(e_i)$ does not depend on the index $i$: it is simply $r(e)_k$, where $k$ is chosen so that $e\in \mathscr{E}^v_k$. So we obtain \begin{equation}\label{bode} r(e_i)=r(e_j) \text{ for all $e$ with $s(e)=v$,} \end{equation}  and thus $v_i=v_j$.
    
    So the map $\phi$ is well-defined at the level of the free monoid $F_{\overline{E}}$. We need to prove that it factors through $T_E$. Again, let us concentrate on the case $k=0$. Let $v\in E^0$. We need to prove that $\phi(v)$ and $\sum_{e\in s^{-1}(v)}{}^1\phi(r(e))$ coincide. On one hand, we have
    \[\phi(v)=v_1,\]
    and on the other
    \[\sum_{e\in s^{-1}(v)}{}^1\phi(r(e))=\sum_{e\in s^{-1}(v)}{}^1 r(e)_{j(e)},\]
    where $j(e)$ is chosen so that $e\in\mathscr{E}^{r(e)}_{j(e)}$. By (\ref{bode}) we have 
    \[\sum_{e\in s^{-1}(v)}{}^1\phi(r(e))=\sum_{e\in s^{-1}(v)}{}^1 r(e_1).\]
    The edges in $F$ which have source equal to $v_1$ are precisely those of the form $e_1$, with $s(e)=v$. So in $T_F$ we have
    \[\sum_{e\in s^{-1}(v)}{}^1\phi(r(e))=v_1,\]
    just as we wanted.
    
    Therefore the map $\phi$ is well-defined.  It is surjective because it follows from the previous argument that $k v_i = k v_1$, for all $v$. In a similar manner one can define the $\mathbb Z$-monoid homomorphism $\psi\colon T_F \rightarrow T_E;\ v_i \mapsto v$. Since  $\psi$ and $\phi$ are inverse of each other, the map $\phi$ is also injective. 
\end{proof}

\begin{example}
The theorem above is not valid for graphs with sinks. Consider the graphs
\[\xymatrix{
E:& \bullet \ar[r] & \bullet & \bullet \ar[l] &&
F:& \bullet \ar[r] &\bullet &\bullet&\bullet \ar[l]
}\]
so that $F$ is the in-split of $E$ obtained by splitting the two arrows with the same range. Then $M_E=\mathbb{N}$ and $M_F=\mathbb{N}\oplus\mathbb{N}$. In particular, $T_E$ and $T_F$ are not isomorphic as $\mathbb{Z}$-monoids.
\end{example}

\subsection*{Move (O): Out-splitting}

The notions dual to those of in-split and in-amalgam are called \emph{out-split} and \emph{out-amalgam}. Given a graph $E=(E^0,E^1,s,r)$, the \emph{transpose graph} is defined as $E^*=(E^0,E^1,r,s)$.

\begin{definition}[{\cite[Definition 6.3.23]{MR3729290}}]\label{def:outsplit}
    A graph $F$ is an \emph{out-split} (\emph{out-amalgam}) of a graph $E$ if $F^*$ is an in-split (in-amalgam) of $E^*$, and we say that $F$ is formed by performing \emph{Move (O)} on $E$.
    
    More specifically, we consider, for every $v\in E^0$ with $s^{-1}(v)\neq 0$, a partition $\left\{\mathscr{E}^1_v,\ldots,\mathscr{E}^{m(v)}_v\right\}$ of $s^{-1}(v)$. The out-split $F$ is formed as follows:
    \[F^0=\left\{v^i \mid v\in E^0,1\leq i\leq m(v)\right\}\cup\left\{v \mid s^{-1}(v)=\varnothing\right\}\]
    \[F^1=\left\{e^j \mid e\in E^1,1\leq j\leq m(r(e))\right\}\cup\left\{e \mid s^{-1}(r(e))=\varnothing\right\},\]
    with source and range maps defined as follows: If $s^{-1}(r(e))\neq\varnothing$, choose $i$ such that $e\in\mathscr{E}^i_{s(e)}$, and set
    \[s(e^j)=s(e)^i,\qquad r(e^j)=r(e)^j,\text{ where } 1\leq j \leq m(r(e)).\]
    If $s^{-1}(r(e))=\varnothing$, set $r(e)$ as the original range of $e$, and $s(e)=s(e)_i$, where $i$ is chosen so that $e\in\mathscr{E}^i_{s(e)}$.
\end{definition}

\color{red} Let $E$ be a row finite graph and $F$ an out-split of $E$. It is known that the Leavitt path algebra $L_\mathsf{k}(E)$ is graded isomorphic to $L_\mathsf{k}(F)$ (\cite{MR2785945} and \cite[Proposition 6.3.25]{MR3729290}). Similarly the algebra  $L_\mathsf{k}(\overline E)$ is graded isomorphic to $L_\mathsf{k}(\overline{F})$. The isomorphism induces an isomorphism between the $K_0$-groups of these algebras and consequently between the positive cones $M_{\overline E}$ and $M_{\overline{F}}$. One can directly show that $M_{\overline E}$ and $M_{\overline{F}}$ are $\mathbb Z$-monoid isomorphic.  Theorem~\ref{hgfgfgggf} now gives that $T_E$ and $T_{F}$ are $Z$-monoid isomorphic.

In the Theorem below, we establish this fact directly on the level of talented monoids, giving yet another evidence that this monoid directly captures the geometry of the graph, without needing to go into the structure of the Leavitt path algebras.

\color{black}

\begin{theorem}\label{inhfgft1} Let $E$ be a row-finite graph. 
    If a graph $F$ is an out-split of the graph $E$ as in Definition \ref{def:outsplit}, then the map
    \[\phi\colon T_E\to T_F,\quad {}^kv\mapsto\begin{cases}\sum_{i=1}^{m(v)}{}^kv^i&\text{ if }v\text{ is not a sink}\\
    {}^kv&\text{ if }v\text{ is a sink}\end{cases}\]
    is a $\mathbb{Z}$-monoid isomorphism.
\end{theorem}
\begin{proof}
    First we need to prove that $\phi$ is well-defined. As usual, let us concentrate in the case $k=0$ in the definition of $\phi$. We need to verify that for every $v\in E^0$ with $s^{-1}(v)\neq \varnothing$, the elements
    \[\sum_{i=1}^{m(v)}v^i\]
    and
    \[\sum_{e\in s^{-1}(v)}\sum_{j=1}^{m(r(e))}{}^1r(e)^j\]
    are equal in $T_F$. But note that $r(e)=r(e^j)$, and the elements $e^j$ of $F$ are precisely the edges of $F$ which have one of the $v_i$'s as its source. So these two elements agree in $T_F$.
    
    We can construct the inverse of $\phi$ explicitly. Define $\psi\colon T_F\to T_E$ on the generators $v^i$ for which $v$ is not a sink as
    \[\psi(v^i)=\sum_{e\in \mathscr{E}^i_v}{}^1r(e),\]
    and $\psi(v)=v$ if $v$ is a sink. We omit the proof that $\psi$ is well-defined, as it uses essentially the same argument as in the second sequence of equalities below.
    
    If $v$ is not a sink of $E$, then in $T_F$ we have
    \[\phi(\psi(v^i))=\sum_{e\in\mathscr{E}^i_v}{}^1\phi(r(e))=\sum_{e\in\mathscr{E}^i_v}\sum_{j=1}^{m(r(e))} {}^1r(e)^j=\sum\left\{{}^1r(e^j):s(e^j)=v^i\right\}=v^i,\]
    so $\psi$ is a right inverse of $\phi$. Conversely, in $T_E$ we have
    \[\psi(\phi(v))=\sum_{i=1}^{m(v)}\psi(v^i)=\sum_{i=1}^{m(v)}\sum_{e\in\mathscr{E}^i_v}{}^1r(e)=\sum_{e\in s^{-1}(v)}{}^1r(e)=v,\]
    where the third equality follows from the sets $\mathscr{E}^i_v$ being a partition of $s^{-1}(v)$. Thus $\psi$ is a left inverse of $\phi$, which is therefore an isomorphism.\qedhere
\end{proof}

We are in a position to use our results to relate the talented monoid to symbolic dynamics. We refer the reader to \cite[\S 7]{MR1369092} for the notion of (strongly) shift equivalent of matrices and the Krieger's dimension group of a matrix (also see \cite{MR561973}). Recall also that a finite graph is called \emph{essential} if it does not have any sinks and sources~\cite[Definition 6.3.11]{MR3729290}.

\begin{proposition}\label{shifteqprop}
We have the following statements. 
\begin{enumerate}
\item Let $E$ be an essential graph and $F$ be a graph obtained from an in-splitting or out-splitting of the graph $E$. Then $T_E$ is $\mathbb Z$-monoid isomorphic to $T_F$.

\item For essential graphs $E$ and $F$, if the adjacency matrices $A_E$ and $A_F$ are strongly shift equivalent then $T_E$ is $\mathbb Z$-monoid isomorphic to $T_F$.

\item For finite graphs $E$ and $F$ with no sinks, if $T_E$ is $\mathbb Z$-monoid isomorphic to $T_F$, then the adjacency matrices $A_E$ and $A_F$ are shift equivalent.

\end{enumerate}

\end{proposition}

\begin{proof}

(1) This follows from Theorems~\ref{inhfgft} and \ref{inhfgft1}.

(2) If $A_E$ is strongly shift equivalent to $A_F$, a combination of the Williams theorem~\cite[Theorem~7.2.7]{MR1369092} and the Decomposition theorem~\cite[Theorem~7.1.2, Corollary~7.1.5]{MR1369092} implies that the graph $F$ can be obtained from $E$ by a sequence of out-splittings, in-splittings and the inverses of these, namely, out-amalgamations, and in-amalgamation. All the graphs which appear in this sequence are essential. Now a repeated application of part (1) gives that $T_E$ is $\mathbb Z$-monoid isomorphic to $T_F$. 

(3) Since $T_E$ is $\mathbb Z$-monoid isomorphic to $T_F$, their group completions are also isomorphic. Thus, there is an order preserving $\mathbb Z[x,x^{-1}]$-module isomorphism $K_0^{\gr}(L_{\mathsf k}(E)) \cong_{\gr} K_0^{\gr}(L_{\mathsf k}(F))$. But this latter isomorphism gives an isomorphism of Krieger's dimension groups $\Delta_E \cong \Delta_F$ (\cite[Corollary 12]{MR3045160}). Thus $A_E$ and $A_F$ are shift equivalent~\cite{MR561973}.
\end{proof}

\subsection{Other graph moves} \label{othermoves}

In this section, we discuss other standard moves that give Morita equivalent Leavitt path algebras, but which are not invariants for talented monoids, and therefore do not give graded Morita equivalent Leavitt path algebras. 

\begin{definition} [{\cite[Definition 6.3.17]{MR3729290}}] Let $E = (E^0, E^1, r, s)$ be a directed graph, and let $v
\in E^0$. Let $v^*$ and $f$ be symbols not in $E^0 \cup E^1$.   We
form the {\em expansion graph} $E_v$ from $E$ at $v$ as follows:
\begin{align*}
E_v^0      &= E^0 \cup \{ v^\ast \} \\
E_v^1      &= E^1 \cup \{ f \} \\
s_{E_v}(e) &= \left\{ \begin{array}{ll}
              v      & \textrm{ if $e = f$} \\
              v^\ast & \textrm{ if $s_E(e) = v$} \\
              s_E(e) & \textrm{ otherwise}
              \end{array}\right. \\
r_{E_v}(e) &= \left\{ \begin{array}{ll}
              v^\ast & \textrm{ if $e = f$} \\
              r_E(e) & \textrm{ otherwise}
              \end{array}\right. \\
\end{align*}
\noindent Conversely, if $E$ and $G$ are graphs, and there exists a
vertex $v$ of $E$ for which $E_v=G$, then $E$ is called a {\em
contraction} of $G$.
\end{definition}

In~\cite[Proposition 6.3.19]{MR1369092}, it was shown that for a finite graph $E$ such that $L_{\mathsf k}(E)$ is simple, the expansion of the graph $E$ produces a Leavitt path algebra Morita equivalent to $L_{\mathsf k}(E)$.  The following example shows that, in general, the graph expansion changes the structure of the talented monoid and the corresponding Leavitt path algebras are not graded Morita equivalent (despite being Morita equivalent).  

Let
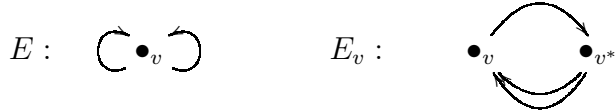
\[\xymatrix{
E:& \bullet_v \ar@(dl,ul) \ar@(dr,ur)  & &
E_v:& \bullet_v \ar@/^1.5pc/[r] & \bullet_{v^*} \ar@/^1.35pc/[l] \ar@/^1.8pc/[l] 
}\]

\bigskip 
The Leavitt path algebra 
$L_{\mathsf k}(E)$ is simple, and thus by~\cite[Proposition 6.3.19]{MR1369092}, 
$L_{\mathsf k}(E_v)$ is Morita equivalent to $L_{\mathsf k}(E)$.  
Note that the period of the graph $E$ is $1$ whereas the period of $E_v$ is $2$. In Theorem~\ref{thm:stronglyconnected} we show that the talented monoids can determine the period of the graphs.
 Thus $T_E$ is not $\mathbb Z$-isomorphism to $T_{E_v}$, and consequently $L_{\mathsf k}(E_v)$ and  $L_{\mathsf k}(E)$ are not graded Morita equivalent. 

However, there are cases where the talented monoid is invariant under the expansion move. Let us give one example here. Let 
\[\xymatrix{
E:& \bullet_v \ar[r] & \bullet & &
E_v:& \bullet_{v} \ar[r] & \bullet_{v^*} \ar[r] & \bullet 
}\]

It is easy to see that the monoids 
$T_E$ and $T_{E_v}$ are both free monoids generated by symbols $v(i)$. Thus $T_E\cong T_{E_v}$. By \cite[Theorem~2]{MR3004584}, we have $L_{\mathsf k}(E)\cong \mathbb M_{2}(K)(0,1)$ and $L_{\mathsf k}(E_v)\cong \mathbb M_{3}(K)(0,1,2)$. These two algebras are graded Morita equivalent, corroborating Conjecture~\ref{conjehfyhtr}. 

We give one more example of an expansion of a graph that does not preserve the talented monoid. We then show directly the associated Leavitt path algebras are not graded Morita equivalent. 
Let

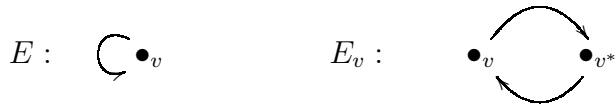
\[\xymatrix{
E:& \bullet_v \ar@(lu,ld) & &
E_v:& \bullet_v \ar@/^1.5pc/[r] & \bullet_{v^*} \ar@/^1.5pc/[l] 
}\]

\bigskip 

It is now easy to use the definition of the talented monoids and directly calculate that $T_E\cong \mathbb N$, whereas $T_{E_v}\cong \mathbb N\oplus \mathbb N$ (see also Theorem~\ref{thm:stronglyconnected}). Indeed, $L_{\mathsf k}(E)\cong \mathsf k [x,x^{-1}]$, with $L(E)_0\cong \mathsf k$, whereas, $L_{\mathsf k}(E_v)\cong \mathbb M_2(\mathsf k[x^2,x^{-2}])(0,1)$, with $L(E_v)_0\cong \mathsf k\oplus \mathsf k.$ (see~\cite[Theorem~2]{MR3004584}).

Contrary to the case of graph 
$C^*$-algebras, the behaviour of Leavitt path algebras under the Cuntz splice remains unknown. It is even said that this is the most compelling unresolved question in the subject of Leavitt path algebras, as it is a test case to the classification of purely infinite simple Leavitt path algebras via the $K_0$-group (see~\cite[\S7]{MR3729290}).
Consider the graph $E_2$ and the graph $E_2^-$ obtained by performing a Cuntz splice to $E_2$.

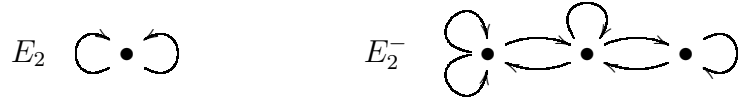
\[E_2 \qquad \xymatrix{ 
		\bullet \ar@(dl,ul) \ar@(dr,ur) 
	} 
	\qquad
	\qquad
	\quad
	\qquad E_2^- \qquad
	\xymatrix{ 
		\bullet \ar@(l,u) \ar@(l,d) \ar@/^0.5em/[r] & \bullet \ar@/^0.5em/[l] \ar@/^0.5em/[r] \ar@(ul,ur) & \bullet \ar@/^0.5em/[l] \ar@(ur,dr)
	}
\]
\smallskip 

The Leavitt path algebra of $E_2$ is the Leavitt algebra $L_2$, and the Leavitt path algebra of $E_2^-$ is often denoted ${L_2}^-$.  It is currently an open question as to whether $L_2$ and ${L_2}^-$ are Morita equivalent. The following argument shows that $L_2$ and ${L_2}^-$ are not graded Morita equivalent and their talented monoids are different. We note that both $E_2$ and $E_2^-$ are strongly connected of period $1$. In Theorem~\ref{thm:stronglyconnected} we show that the talented monoids can determine the period of the graphs. However this example shows that there are still some other properties of the graphs that can possibly be captured by the talented monoids. 

Suppose that $T_{E_2}$ is 
$\mathbb Z$-isomorphic to $T_{E_2^-}$. By Proposition~\ref{shifteqprop}(3), the adjacency matrices of $E_2$ and $E_2^-$ are shift equivalent. Now~\cite[Exercise 7.4.4]{MR1369092}, for $p(t)=1-t$, gives that $\det(1-E_2)=\det(1-E_2^-)$. However, we know that these two determinants are not the same (as the Cuntz splice is conceived to change the sign of the determinant of the adjacency matrix). Thus $T_{E_2}$ and $T_{E_2^-}$ are not $\mathbb Z$-isomorphic and consequently their Leavitt path algebras are not graded Morita equivalent.

\section{Cycle properties of a graph and the talented monoid}\label{sec:extremecycles}

Recall from Section~\ref{premini} that we can distinguish several kinds of cycles in graphs. In \cite[Proposition~4.2]{MR4040730} the cycles with and without exits were described in the talented monoid: In a graph $E$, there is a cycle with no exit if and only if there is an $x\in T_E$ such that ${}^k x = x$ for some $k\neq 0$. On the other hand, there is cycle with an exit if and only if there is $x\in T_E$ such that ${}^k x < x$, for some 
$k \in \mathbb N$.

In this section we describe extreme cycles in a graph in terms of its associated talented monoid.

\begin{proposition}\label{extermeme}
Let $E$ be a row-finite graph and $T_E$ its talented monoid. Then the following are equivalent: 

\begin{enumerate}
\item The graph $E$ has an extreme cycle.

\item There exists $x\in T_E$ such that ${}^k x <x$ for some $k\in \mathbb N$ and if $0\neq y \leq \sum_i{}^{r_i} x$ for certain $r_i\in\mathbb{Z}$ then $x \leq \sum_j {}^{s_j} y$ for certain $s_j\in\mathbb{Z}$.

\item There exists $x\in T_E$ such that ${}^k x <x$ for some $k\in \mathbb N$ and $\langle x \rangle$ is a simple $\mathbb Z$-order ideal. 
\end{enumerate}

\end{proposition}
\begin{proof}
Since the $\mathbb{Z}$-order ideal generated by an element $x\in T_E$ consists of the elements $y$ such that $y\leq\sum_i {}^{r_i}x$ for $r_i\in\mathbb{Z}$, and similarly for $y$, it follows that (2) and (3) are equivalent.

\medskip

(1) $\Rightarrow$ (2). First assume that the graph $E$ has an extreme cycle $\mathfrak{c}=e_1e_2\cdots e_k$. Let $x_i=s(e_i)$, for $1\leq i\leq k$ and $x_{k+1}=x_1=r(e_k)=s(e_1)$. 

Set $x=x_1=x_1(0)$. In $T_E$, we have
\[x=x_1(0)\geq x_2(1)\geq \cdots \geq x_k(k-1) \geq x_1(k), \]
because there is an edge from $x_i$ to $x_{i+1}$. We obtain $x\geq {}^kx$. Since $\mathfrak{c}$ has an exit then one of these inequalities is strict, so $x> {}^kx$. It remains to prove that $\langle x\rangle$ is a simple $\mathbb{Z}$-order ideal.

Suppose that $0\neq y\leq \sum_i {}^{r_i}x$. We have $y+z=\sum_i {}^{r_i}x$ in $T_E\cong M_{\overline{E}}$ for some $z$. Confluence Lemma~\ref{confuu} implies that there exists $c=\sum c_j(n_{c,j})$ in $F_{\overline{E}}$ such that $\sum_i{}^{r_i}x\to c$ and $y+z\to c$. Here, $c_j\in E^0$ and $n_{c,j}\in\mathbb{Z}$.

Let us expand $y=\sum y_i(n_{y,i})$, where $y_i\in E^0$ and $n_{y,i}\in\mathbb{Z}$. The vertex $y_1(n_{y,1})$ of $\overline{E}$ is in the representation of $y+z$, as an element of $F_{\overline{E}}$. By Proposition \ref{prop:flowpath}(1), $y_1(n_{y,1})$ flows to some $c_j(n_{c,j})$ in $\overline{E}$, which implies that $y_1$ flows to $c_j$. Up to reordering, we may assume that $y_1$ flows to $c_1$. But then, since $\sum_i{}^{r_i}x\to c$, item (2) of that same proposition also implies that the vertex $x$ flows to $c_1$.

By the paragraph after Proposition \ref{prop:flowpath}, we can find paths $\mu$ and $\lambda$, starting at $y_1$ and at $x$, respectively, such that $r(\mu)=r(\lambda)=c_1$.

Since $\mathfrak{c}$ is an extreme cycle, then there exists a path $\beta$ from $c_1$ to some vertex of $\mathfrak{c}$, which we can assume to be $x$ (extending $\beta$ along $\mathfrak{c}$ if necessary).

We can now construct a path from $y_1(n_{y,1})$ to $c_1(n_{y,1}+|\mu|)$. Namely, if $\mu=\mu_1\cdots\mu_{|\mu|}$, where $\mu_j\in E^1$, we have the path
\[y_1(n_{y,1})\overset{(\mu_1,n_{y,1})}{\longrightarrow}r(\mu_1)(n_{y,1}+1)\overset{(\mu_2,n_{y,1}+1)}{\longrightarrow}\cdots\overset{(\mu_{|\mu|},n_{y,1}+|\mu|-1)}{\longrightarrow}r(\mu_{|\mu|})(n_{y,1}+|\mu|)=c_1(n_{y,1}+|\mu|).\]
Similarly, there is a path from $c_1(n_{y,1}+|\mu|)$ to $x(n_{y,1}+|\mu|+|\beta|)$. So we obtain a path from $y_1(n_{y,1})$ to $x(p)$, for appropriate $p$. So in $M_{\overline{E}}\cong T_E$ we obtain $y_1(n_{y,1})\geq x(p)$, or equivalently ${}^{-p}y_1(n_{y,1})\geq x(0)$.
We conclude that $x\leq {}^{-p}y.$

\medskip

(2) $\Rightarrow$ (1). Let $x$ be as in statement (2). 
\medskip

\textbf{Claim.} No sink appears in any representation of $x$.

Suppose otherwise, that $s$ is a sink and $x=s(i)+y$ for some $y\in T_E$ and some $i\in \mathbb Z$. Then $s(i) \leq x$ and by assumption we obtain $x +t = \sum_j s(p_j)$ for some $t\in T_E$ and $p_j\in\mathbb{Z}$. All the vertices $s(p_j)$ are sinks in $\overline{E}$, so they ``do not flow''. This is to say that, by Proposition \ref{prop:flowpath}, if $c\in F_{\overline{E}}$ and $\sum_j s(p_j)\to c$ then $c=\sum_j s(p_j)$ in $F_{\overline{E}}$.

By the Confluence Lemma~\ref{confuu} we have $x+t \rightarrow \sum_j s(p_j)$. This implies that one can write $x=\sum_i s(q_i)$ in $T_E$ for some subcollection $\left\{q_i\right\}_i\subseteq\left\{p_j\right\}_j$. Since ${}^k x <x$ for some $k>0$, we can choose $k$ large enough so that all the shifts in ${}^k x$ are larger than the shifts in $x$. The inequality ${}^k x < x$ yields ${}^k x +\widetilde{t}=x$ for some $\widetilde{t}$, and the same argument implies that ${}^k x +\widetilde{t} \rightarrow x$. However, let $Q$ be the largest among all $q_i$. Then $s(Q+k)$ appears in the representation of ${}^k x+\widetilde{t}$, so Proposition \ref{prop:flowpath} implies that $s(Q+k)$ flows to some $s(q_i)$, a contradiction.

\medskip
Let us expand $x=\sum_i x_i(n_{x,i})$. By the claim above, none of the $x_i$ are sinks. Letting these vertices flow, we can rewrite all of the terms $x_i(n_{x,i})$ ``at the same level'', that is, $x=\sum x_i(n_x)$ for a single number $n_x$, so that ${}^kx=\sum x_i(n_x+k)$.

By hypothesis, we have ${}^kx<x$. By the Confluence Lemma~\ref{confuu}, we can find $c,d\in F_{\overline{E}}$ such that $x\to c+d$ and ${}^kx\to c$. Again, $c+d$ is simply another presentation of $x$ in $T_E$, so the vertices which appear in any presentation of $c$ and $d$ are not sinks, and we can let them flow as much as necessary and assume all of them appear at the same level as well: $c=\sum c_i(N)$ and $d=\sum d_i(N)$, where $N>n_x+1$. Note that, since we simply let the vertices flow, the relations $x\to c+d$ and ${}^kx\to c$ are still valid.

We have ${}^kx\to c$. Let all of the vertices appearing in $c$ flow to the level $N+k$, and consider the element $\overline{c}$ of $F_{\overline{E}}$ which we obtain in this manner.

From the relation $x\to c+d$, we obtain ${}^kx\to {}^kc+{}^kd$, and all vertices of ${}^kc+{}^kd$ are at level $N+k$ as well. This means that
\[{}^kx\to {}^kc+{}^kd\quad\text{and }{}^kx\to\overline{c},\]
and all vertices of ${}^kc$, ${}^kd$ and $\overline{c}$ are at the level $N+k$. This is only possible if ${}^kc+{}^kd=\overline{c}$, that is, that ${}^kc+{}^kd$ is what we obtain when we let $c$ flow by $k$ levels.

In $T_E$ we have $x=c+d$ and ${}^kx=c<x$. Thus $d\neq 0$, so ${}^kd\neq 0$ as well. Thus the number of vertices (of $\overline{E}$) which appear in the presentations of ${}^kc+{}^kd$ is strictly greater than that of ${}^kc$, which is the same as the one of $c$. Therefore, at least one of the vertices $\bar{c}_1(N)$ in the presentation of $c$ will be the source of at least two distinct arrows (this is called a \emph{bifurcation vertex}).

Now, $c$ flows to ${}^kc+{}^kd$, and $\bar{c}_1(N+k)$ is in the presentation of ${}^kc+{}^kd$. Proposition \ref{prop:flowpath} implies that there is some $\bar{c}_2(N)$ among the vertices of the representation $c$ and a path from $\bar{c}_2(N)$ to $\bar{c}_1(N+k)$, and in particular there is a path from $\bar{c}_2$ to $\bar{c}_1$.

Repeat this procedure and construct a path
\[\cdots \bar{c}_3\to \bar{c}_2\to \bar{c}_1\]
After some point, one of the $\bar{c}_M$ will have already appeared as a previous $\bar{c}_j$, so in fact we have constructed a cycle $\mathfrak{c}\colon \bar{c}_M\to\cdots\to \bar{c}_j$. Take the smallest such $M\geq 2$ and associated $j<M$. If $\bar{c}_j=\bar{c}_1$, then this cycle has an exit by our choice of $\bar{c}_1$. If not, then $\bar{c}_j$ has paths pointing both to $\bar{c}_{M-1}$ and to $\bar{c}_{j-1}$, which are different by the minimality of $M$. In any case, this cycle has an exit.

As a matter of convenience, let us rewrite this cycle as $\mathfrak{c}\colon v_1\to\cdots\to v_n=v_1$, where the $v_i$ are vertices.

We just need to prove that $\mathfrak{c}$ is extreme. Let $\alpha$ be any path starting at $v_1$. Then in $T_E$ we have
\[0\neq r(\alpha)(N+|\alpha|)\leq v_1(N)\leq c\leq c+d=x\]
The hypotheses on $x$ give us numbers $k_p$ such that
\[v_1(N)\leq x\leq\sum_p r(\alpha)(N+|\alpha|+k_p)\]
By the Confluence Lemma~\ref{confuu}, there are $t,w\in F_{\overline{E}}$ such that
\[v_1(N)\to t\qquad\text{and}\qquad\sum_p r(\alpha)(N+|\alpha|+k_p)\to t+w.\]
The presentation of $t$ will necessarily have an element of the form $v_j(M)$, because $v_1\cdots v_n$ is a cycle and $v_1(N)\to t$. So this same term $v_j(M)$ is also in the presentation of $t+w$. Proposition \ref{prop:flowpath} implies that there is $p$ and a path from $r(\alpha)(N+|\alpha|+k_p)$ to $v_j(M)$. In particular there is a path from $r(\alpha)$ back to the vertex $v_j$.

This proves that $\mathfrak{c}$ is extreme.\qedhere
\end{proof}

We can now use this description of extreme cycles to complement the results of \cite{MR4040730}. First we recall how one can describe the cycles with no return exit.

\begin{proposition}\label{propnongra}
Let $E$ be a row-finite graph and $T_E$ its talented monoid. Let ${\mathsf k}$ be a field. Then the following are equivalent: 

\begin{enumerate}
\item The graph $E$ has a cycle with no return exit;

\item There exists an order ideal $I$ of $T_E$ such that $T_E/I$ has a periodic element;

\item The Leavitt path algebra $L_{\mathsf k}(E)$ has a non-graded ideal.
\end{enumerate}

\end{proposition}
\begin{proof}
This follows from Proposition 5.2 and its proof  in \cite{MR4040730}.
\end{proof}

We say that two extreme cycles of a graph are \emph{disjoint} if there is no path connecting a vertex from one cycle to a vertex of the other cycle. We say that two extreme cycles are related if they are not disjoint. This defines an equivalence relation between extreme cycles. By  the ``collection of disjoint extreme cycles'' we mean the partition of the set of extreme cycles under this relation. The collections of disjoint cycles with no exits and of disjoint line points are regarded similarly. These will play a main role in Section \ref{sec:perioifhf}. For now, we determine these types of cycles in terms of the talented monoid.

\begin{proposition}\label{azul}
Let $E$ be a row-finite graph.

\begin{enumerate}
\item There is a one to one correspondence between disjoint extreme cycles and simple $\mathbb Z$-order ideals $\langle x \rangle$ of $T_E$ with ${}^k x <x $, for some $k>0$.

\item There is a one to one correspondence between disjoint cycles with no exits and simple $\mathbb Z$-order ideals $\langle x \rangle$ of $T_E$ with ${}^k x = x $, for some $k>0$.

\item There is a one to one correspondence between disjoint line points and simple $\mathbb Z$-order ideals $\langle x \rangle$ of $T_E$ with $x$ and ${}^i x$ not comparable for any $i\neq 0$. 

\end{enumerate}
\end{proposition}
\begin{proof}
(1) Let $\mathfrak{c}$ be an extreme cycle with $v\in \mathfrak{c}^0$. By the proof of $(1)\Rightarrow(2)$ in Proposition~\ref{extermeme}, considering $v\in T_E$, we have ${}^k v <v$ and $\langle v \rangle$ is a simple $\mathbb Z$-order ideal of $T_E$. Furthermore, if $\mathfrak{d}$ represents a disjoint extreme cycle to $\mathfrak{c}$, choosing a vertex $w$ on $\mathfrak{d}$, we get a simple $\mathbb Z$-order ideal $\langle w \rangle$ of $T_E$. If $\langle v \rangle=\langle w \rangle$ then $w \leq \sum k_i {}^i v$. Since $v$ is on an extreme cycle, a similar argument as in the proof of Proposition~\ref{extermeme} shows that $w$ is connected to $v$ which is not the case. On the other hand if there is $x\in T_E$ such that ${}^k x < x$ and $\langle x \rangle$ is a simple $\mathbb Z$-order ideal then part (2) of Proposition~\ref{extermeme} guarantees that there is an extreme cycle in $E$. Putting these together, we have established a one-to one correspondence. 

(2) Let $C$ be the set of all cycles without exits in $E$ and let $A$ be the set of simple $\mathbb Z$-order ideals of the form $\langle x \rangle$ such that ${}^kx=x$. For a cycle $c \in C$, denote by $c_v$ a vertex on the cycle (there is no need to fix this base vertex). Then 
$\langle c_v \rangle$ is in the set $A$. One can show that the map defined from $C$ to $A$ as above is bijective (see \cite[Lemma 5.6]{MR4040730}, in particular the proof of Lemma 5.6(iv)).

(3) is similar to the argument of (2) using \cite[Lemma 5.6]{MR4040730}. 
\end{proof}

\begin{example} Consider the following two graphs: 
\[
\xymatrix{
   && \bullet  \ar@(ru,rd)\\
  E_1:\hspace{-20pt}&\bullet  \ar[ru] \ar[rd]\\
  && \bullet  \ar@(ru,rd)
 }\qquad\qquad
\xymatrix{
   && \bullet  \ar@(ru,rd)\\
  E_2:\hspace{-20pt}&\bullet  \ar[ru] \ar[rd]\\
  && \bullet 
 }\]
\medskip 

Notice that $M_{E_1}\cong M_{E_2}$, but $E_1$ has two cycles with no exits, whereas $E_2$ has only one. This example indicates that in general one can not formulate a statement similar to Proposition 5.3 for $M_E$ (in particular the example shows that the analogue Proposition~\ref{azul}(2) and (3) do not hold for $M_E$). 
\end{example}

In the next section we show that not only the collection of extreme cycles are preserved by the talented monoid, but also the periods of the extreme cycles are also captured by this invariant.

The proposition above, along with \cite[Corollary 5.1]{MR4040730}  makes it clear why the simple row-finite Leavitt path algebras are either purely infinite simple or simple ultramatricial algebras. For, suppose $L_{\mathsf k}(E)$ is simple. Then  we have 
$T_E=\langle x \rangle$ and either ${}^k x < x$, for some $k\in \mathbb Z$, or they are not comparable (the case ${}^k x = x$ gives a non-simple but graded simple algebra).

\section{Primary colours of Leavitt path algebras and the talented monoids}\label{sec:perioifhf}

The theory of Leavitt path algebras includes well-known, but at the same time rather distinct, classes of algebras. There are three ``extreme cases'' of graphs, which correspond to the so-called ``primary colours'' of Leavitt path algebras as described in \cite{MR3729290}. These are line points, cycles without exits, and extreme cycles, which will be considered below.

Pask and Rho studied the notion of period of the graph in relation with the graph $C^*$-algebras in \cite{MR2018239}. They showed that for a finite strongly connected graph $E$, the covering graph $\overline{E}$ admits a partition into $d(E)$ disjoint isomorphic connected subgraphs $E_0,\ldots,E_{d(E)-1}$.  The notion of period also appears in both the theory of symbolic dynamics and Markov chains, where it gives cyclic structures in the corresponding theories (see~\cite{MR1369092}).

In this section we analyse how the period of a finite graph is encoded in its talented monoid. Along the way, we also give a different proof that the period of all vertices of a finite strongly connected graph are the same.

Given $v\in E^0$, recall from \S\ref{balconyfreshwater} the notion $[v]$ of the order ideal of $T_E$ generated by $v$. 

\begin{proposition}\label{prop:stronglyconnectedimpliesdecomposition}
    Let $E$ be a strongly connected finite graph, $v\in E^0$ and let $d$ be the period of $v$. Then
    \begin{enumerate}
        \item For all $0<i<d$, we have $[v]\cap{}^i[v]=\left\{0\right\}$.
       
        \item $[v]={}^d[v]$.
        
         \item $[v]$ is a simple order ideal.
        
        \item $T_E=[v]\oplus{}^1[v]\oplus\cdots\oplus{}^{d-1}[v]$.
    \end{enumerate}
\end{proposition}
\begin{proof}
Note that ${}^i[v]=[{}^iv]$ for all $i\in\mathbb{Z}$.

\medskip

(1) Let $i\geq 0$. Suppose that ${}^jw\in[{}^iv]$, where $w\in E^0$ and \(j\in\mathbb{Z}\). Since $E$ is strongly connected, there is a path $\alpha$ with $v=s(\alpha)$ and $w=r(\alpha)$. First, we prove that $i-j+|\alpha|$ is a multiple of $d$.
    
    Take \(n\in\mathbb{N}\) such that \({}^jw\leq n({}^iv)\). By the Confluence Lemma~\ref{confuu}, there exist $a,b\in F_{\overline{E}}$ such that ${}^jw\to a$ and $n({}^iv)\to a+b$. 
    Letting the vertices of $a$ flow for as long as necessary, and since $E$ has no sources (as it is strongly connected), we can assume that all vertices in the presentation of $a$ are at the same ``level'', i.e., that $a=\sum_k\left({}^pa_k\right)$ for some $p$ sufficiently large, for certain vertices \(a_k\), and that $v=a_k$ for some \(k\). Since \(n({}^iv)\to a+b\), then by Proposition \ref{prop:flowpath}(2) there is a path from ${}^iv$ to ${}^pv$ in \(\overline{E}\), which corresponds to a cycle containing $v$ of length $p-i$ in \(E\). Thus $p-i$ is a multiple of $d$.
    
    Similarly there is a path of length $p-j$ from $w$ to $v$, so concatenating with $\alpha$ we obtain a cycle of length $p-j+|\alpha|$ containing $v$. So $p-j+|\alpha|$ is also a multiple of $d$. Therefore $i-j+|\alpha|$ is a multiple of $d$.
    
    Now we can prove that $[v]$ and ${}^i[v]$ have trivial intersection for $0<i<d$. Suppose that this was not the case, and let $0\neq x\in[v]\cap[{}^iv]$. Consider any term ${}^j w$, with $w$ a vertex, which appears in a representation of $x$, and let $\alpha$ be a path connecting $v$ to $w$. As we have seen above, $0-j+|\alpha|$ and $i-j+|\alpha|$ are both multiples of $d$, so $i$ is also a multiple of $d$, a contradiction.
    Therefore, $[v]\cap{}^i[v]=\left\{0\right\}$ for $0<i<d$.
    
    \medskip
    
(2) To prove that $[v]=[{}^dv]$, it suffices to show that $v\leq k_1 \left({}^dv\right)$ and ${}^dv\leq k_2 (v)$ for some $k_1,k_2>0$.
    
    Consider the power set $P(E^0)$ of $E^0$, and let $\phi\colon P(E^0)\to P(E^0)$ be given by $\phi(A)=r(s^{-1}(A))$.
    
    Let $\alpha$ be a cycle starting and ending at the vertex $v$. Consider the sequence
    \[A_0=\left\{v\right\},\qquad A_{n+1}=\phi^{|\alpha|}(A_n),\quad n\geq 0.\]
    By our choice of $\alpha$, we have $A_0\subseteq A_1$, so recursively we obtain $A_n\subseteq A_{n+1}$. Since $P(E^0)$ is finite, the sequence $\left\{A_n\right\}_n$ eventually stabilises. Consider $k$ such that $A_k=A_{k+1}$.
    
    For every $n$, we may rewrite $v$ in $T_E$ as
    \[v=\sum_{w\in A_n}\omega_n(w)\left({}^{|\alpha|n}w\right),\]
    for certain strictly positive `` weights'' $\omega_n>0$. This is to say that, up to shifts, the elements of $A_n$ are precisely the terms which appear in the representation of $v$ at step $|\alpha|n$.
    
    Using this at $n=k$ and $n=k+1$, we obtain
    \[^{|\alpha|}v={}^{|\alpha|}\left(\sum_{w\in A_k}\omega_k(w)\,\left({}^{|\alpha|k}w\right)\right)=\sum_{w\in A_k}\omega_k(w)\,\left({}^{|\alpha|(k+1)}w\right),\]
    and
    \[v=\sum_{w\in A_{k+1}}\omega_{k+1}(w)\, \left({}^{|\alpha|(k+1)}w\right).\]
    Since $A_k=A_{k+1}$ and all weights $\omega_k(w)$ are strictly positive, we conclude that
    \[v\leq \left(\sum_{w\in A_k}\omega_{k+1}(w)\right)\,\left({}^{|\alpha|}v\right).\]
    
    Since the period of $v$ is $d$, by Bézout's Lemma, there exist cycles $\alpha_1,\ldots,\alpha_n$, all containing $v$, and integers $p_1,\ldots,p_n$, such that $\sum_i p_i|\alpha_i|=d$. For each $i$, the argument above yields $N_i$ such that $v\leq N_i \, \left({}^{|\alpha_i|}v\right)$.
    
    For $p_i>0$, we have $v\leq N_i\left({}^{|\alpha_i|}v\right)\leq \left(N_i^2\right)\left({}^{2|\alpha_i|}v\right)\leq\cdots\leq\left(N_i^{p_i}\right)\left({}^{p_i|\alpha_i|}v\right)$, so
    \begin{equation}\label{ghghgh}
      v\leq\left(\prod_{p_i>0}N_i^{p_i}\right){}^{\left(\sum_{p_i>0}p_i|\alpha_i|\right)}v.
    \end{equation}
    
    For $p_i<0$, we have ${}^{-p_i|\alpha_i|}v\leq v$, so
    ${}^{\left(-\sum_{p_i<0}p_i|\alpha_i|\right)}v\leq v,$, that is, 
    \begin{equation}\label{ghghgh1}
    v\leq{}^{\left(\sum_{p_i<0}p_i|\alpha_i|\right)}v.
    \end{equation}
    
    Putting (\ref{ghghgh}) and (\ref{ghghgh1}) together, we conclude that
    \[v\leq\left(\prod_{p_i>0}N_i^{p_i}\right)\left({}^d v\right).\]
    
    This proves that $[v]\subset[{}^dv]$. Similarly, we prove that ${}^dv\leq\left(\prod_{p_i<0}N_i^{-p_i}\right)v$, so $[{}^dv]\subset [v]$.
    
    \medskip
    
    (3) Now we prove that $[v]$ is a simple order ideal. Let $0\neq x\in[v]$. Since the graph $E$ is strongly connected, letting the vertices in a given representation of $x$ flow as long as necessary, we can find $j$ large enough such that ${}^jv\leq x\in[v]$, so that ${}^jv\in[v]\cap[{}^jv]$. Items (1) and (2) imply that $j$ is a multiple of $d$, so
    \[[v]=[{}^jv]\subseteq[x]\subseteq[v].\]

\medskip

(4) We can now prove that $T_E=\bigoplus_{i=0}^{d-1}{}^i[v].$ By items (1) and (3) and the fact that ${}^i [v]$ are simple order ideals, we have that $[v], {}^1[v], \cdots, {}^{d-1} [v]$ constitute a direct summand, and thus $\bigoplus_{i=0}^{d-1}{}^i[v] \subset T_E$.

On the other hand, given $w\in E^{0}$ and $j\in\mathbb{Z}$, consider a path $\alpha$ from $v$ to $w$, so that ${}^{|\alpha|}w \leq v$, i.e., ${}^jw\leq {}^{j-|\alpha|}v$. Then ${}^jw\in{}^i[v]$, where $i$ is the remainder of the division of $j-|\alpha|$ by $d$, and hence $\bigoplus_{i=0}^{d-1}{}^i[v] = T_E$..
\end{proof}

We are in a position to prove the main theorem of this section. 

\begin{theorem}\label{thm:stronglyconnected}
    Let $E$ be a finite graph with no sources and $d\in\mathbb{N}$. The following are equivalent:
    \begin{enumerate}[label=(\arabic*)]
        \item $E$ is strongly connected and the period of all vertices of $E$ is $d$;
        
        \item $E$ is strongly connected and the period of at least one vertex of $E$ is $d$;
        
        \item There exists a simple order ideal $I$ of $T_E$ such that ${}^dI=I$ and
        \[T_E=I\oplus{}^1I\oplus\cdots\oplus{}^{d-1}I.\]
    \end{enumerate}
    Moreover, the decomposition of $T_E$ as in (3) is unique up to permutation; namely, for every vertex $v$ there is an $i\in \mathbb N$ such that $I=[^{i}v]$.
\end{theorem}

\begin{proof}
    The implication (1)$\Rightarrow$(2) is trivial, whereas (2)$\Rightarrow$(3) follows from Proposition~\ref{prop:stronglyconnectedimpliesdecomposition}.
    We are left to prove the implication (3)$\Rightarrow$(1). Let $I$ be as in (3). We start by proving that $E$ is strongly connected.
    
    \textbf{Claim 1.} Up to a shift, $I=[w]$ for some vertex $w$.
    
    Indeed, if $x$ is a nonzero element of $I$ then $x={}^jw+\widetilde{x}$, for some vertex $w$ and $j \in \mathbb Z$. Since $I$ is an order ideal, ${}^j w \in I$ and since it is simple, $I=[{}^jw]$. Shifting $I$ if necessary, we obtain $I=[w]$.
    
    \textbf{Claim 2.} If $I=[w]$ and $w$ is flowed into from the vertex $u$, then $w$ also flows to $u$.
    
    Since $E$ is finite and has no sources, we can find a cycle $u_0\to u_1\to\cdots\to u_n=u_0$ such that $u_0$ flows to $u$. It then suffices to prove that $w$ flows to $u_0$.
    
    Since $T_E=\bigoplus_{i=0}^{d-1}{}^iI$, let us rewrite $u_0=\sum_{i=0}^{d-1}m_i$, where $m_i\in[{}^iw]$. Consider numbers $N_i$ such that $m_i\leq N_i\,\left({}^i w\right)$. By the Confluence Lemma~\ref{confuu} and Proposition \ref{prop:flowpath}, we can find $c_i\in F_{\overline{E}}$ such that $m_i\to c_i$ in such a way that every vertex which appears in the representation of $c_i$ can be flowed into from ${}^iw$. In $T_E$ we have $m_i=c_i$, so $u_0=\sum_i c_i$. Applying the Confluence Lemma~\ref{confuu} again, we obtain $b_0,b_1,\ldots,b_{d-1}$ such that $u_0\to\sum_i b_i$ and $c_i\to b_i$. Since $u_0$ belongs to the cycle $u_0\to\cdots\to u_{m-1}\to u_0$, then at least one vertex of the form ${}^ju_j$ will appear in the representation of one of the $b_i$.
    
    By construction, all the vertices which appear in the representation of $c_i$ can be flowed into from ${}^iw$. In particular, $w$ flows to $u_j$, just as we wanted.
    
    \textbf{Claim 3.} If $I=[w]$ and $w$ flows to $u$, then $u$ also flows to $w$. Moreover, there exists $i$ such that $I=[{}^iu]$.
    
    Given $u$ and $w$ as in the hypothesis of Claim~3, the same argument as in the proof of Claim 1 shows that, up to a shift, $I=[u]$. Applying Claim 2., with the roles of $w$ and $u$ exchanged, yields the desired claim.
    
    \medskip
    
    We can now proceed to prove that $E$ is strongly connected. Using Claim 1., assume that $I=[w]$. Let $v$ be any vertex of $E$. Choose $i$ such that $[v]\cap{}^iI\neq\varnothing$. This implies that $v$ and $w$ flow to a common vertex $u$. By Claim 3., $u$ also flows to $w$, so $v$ flows to $w$ as well. By Claim 2., $w$ also flows to $v$.
    
    Now we need only to prove that any vertex of $E$ has period $d$. Let $v$ be any vertex of $E$. Since $E$ is strongly connected, we apply Claims 1. and 3. above to conclude that $I=[v]$ (up to a shift).
    
    By Proposition \ref{prop:stronglyconnectedimpliesdecomposition}, we have $T_E=\bigoplus_{i=0}^{d'-1}{}^iI$, where $d'$ is the period of $v$. But also $T_E=\bigoplus_{i=0}^{d-1}{}^iI$. This is only possible if $d=d'$, the period of $v$.\qedhere
\end{proof}

Next we characterize strongly connected graphs that satisfy Condition~(L) in terms of the talented monoid. Note that a commutative semigroup $S$ is a group if and only if for every \(a,b\in S\), there exists \(x\) such that \(ax=b\). If \(S=M\setminus\left\{0\right\}\) for a (commutative) monoid \(M\), this is equivalent to say that $x\leq y$ for all $x,y\in S$. Graph monoids and talented monoids have this property.

\begin{theorem}\label{thm:irredpropldecomposition}
    Let $E$ be a finite graph with no sources. The following are equivalent:
    \begin{enumerate}[label=(\arabic*)]
        \item $E$ is strongly connected and has Condition~(L);
        
        \item $M_E\setminus\left\{0\right\}$ is a group;
        
        \item There exists an order ideal $I$ of $T_E$ and $d\in\mathbb{N}$ such that $I\setminus\left\{0\right\}$ is a group, ${}^dI=I$ and
        \[T_E=I\oplus{}^1 I\oplus\cdots\oplus{}^{d-1}I.\]
        In this case, $d$ is the period of $E$.
        
        \item $L_{\mathsf{k}}(E)$ is purely infinite simple.
    \end{enumerate}
\end{theorem}

\begin{remark} Condition (1) above can be seen to be equivalent to (i)--(iii) of \cite[Theorem~3.1.10]{MR3729290} (using Lemma 2.9.6 of the same book), which along with \cite[Proposition~6.1.12]{MR3729290} proves the equivalence between (1) and (2). The direct proof of (1)$\iff$(2) that we give below uses only the geometry of the graph and the associated monoids.
\end{remark}

\begin{proof}[Proof of Theorem~\ref{thm:irredpropldecomposition}]

    (1) $\Rightarrow$(2) Since the graph monoid $M_E$ is conical, $M_E\setminus\left\{0\right\}$ is a subsemigroup of it. We show that for any 
    $x,y \in M_E\setminus\left\{0\right\}$ we have $x\geq y$, which in turn implies that $M_E\setminus\left\{0\right\}$ is a group. 
     Let $x=x_1+\cdots+x_n$ and $y=y_1+\cdots+y_m$ in $M_E\setminus\left\{0\right\}$, where $x_i$ and $y_j$ are vertices of $E$.
    
    Consider any cycle $\mathfrak{c}=c_1c_2\cdots c_p$ in $E$ (where the $c_i$ are edges). Since $E$ has condition (L), $\mathfrak{c}$ has an exit edge, call it $t$. We can assume that $s(c_1)=s(t)$. Since $E$ is strongly connected then, in $M_E$, any two vertices \(u,v\) satisfy \(u\leq v\), so in particular
    \[x_1\geq s(c_1)\geq r(c_1)+r(t)\geq x_1+x_1.\]
    Iterating the inequality above $m$ times, we conclude, as desired, that
    \[x\geq x_1\geq mx_1\geq y_1+\cdots+y_m=y,\]
    as desired.
    
    (2)$\Rightarrow$(1) Assume that $M_E\setminus\left\{0\right\}$ is a group. We first prove that the graph $E$ is strongly connected. Let $u,v$ be vertices of $E^0$. Since $E$ has no sources and is finite, take a cycle $\mathfrak{c}=c_1\cdots c_n$ such that $s(c_1)$ flows to $v$. Since $M_E\setminus\left\{0\right\}$ is a group, all elements are comparable and thus $u\geq s(c_1)$. By the Confluence Lemma~\ref{confuu}, there exist $x,y\in F_E$ such that $s(c_1)\to x$ and $u\to x+y$. But the vertex $s(c_1)$ belongs to the cycle $\mathfrak{c}$, so at least one of the terms of $x$ has to be a $s(c_j)$, for some $j\in\{1,..,n\}$. Proposition~\ref{prop:flowpath} implies that $u$ flows to $s(c_j)$, so it also flows to $s(c_1)$ and to $v$.
    
    Next we prove that $E$ has condition (L). Suppose that this were not the case. Then, as we already know that $E$ is strongly connected, all edges of \(E\). lie on a single cycle $c_1\cdots c_n$, where $s(c_1)=r(c_n)$ and all the edges $c_i$ have distinct sources.
    
    In $M_E$, we have $s(c_1)\geq 2s(c_1)$. By the Confluence Lemma~\ref{confuu}, there exist $x,y\in F_E$ such that $2s(c_1)\to x$ and $s(c_1)\to x+y$. However, $s(c_1)$ only flows to $s(c_2)$, which only flows to $s(c_3)$, etc., so $x+y$ is actually a single vertex of $E$. But since \(2s(c_1)\to x\) then $x$ has to a sum of at least two vertices in $F_E$, a contradiction. Therefore, $E$ has condition (L).
    
    (2)$\iff$(3) Note that if $I$ is an order ideal of $T_E$ and $I\setminus\left\{0\right\}$ is a group then $I$ is simple. So the decomposition of item (3) -- when it exists -- is the same as the one in Theorem \ref{thm:stronglyconnected}(3).
    
    Since $M_E$ is the quotient of $T_E$ obtained by identifying elements and their shifts, the forgetful homomorphism $T_E\to M_E$ (see \eqref{sungftgdtd}) restricts to an isomorphism $I\to M_E$. In particular, $I\setminus\left\{0\right\}$ is a group if and only if $M_E\setminus\left\{0\right\}$ is a group.
    
    (4)$\iff$(2) follows from \cite[Proposition 6.1.12]{MR3729290}.\qedhere
\end{proof}

\begin{example}\label{example1}
Consider the following graphs  from Example~\ref{exampleR}:

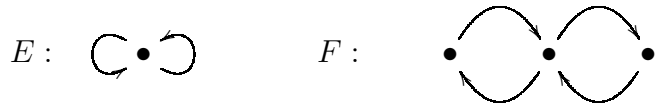
\[\xymatrix{
E:& \bullet \ar@(rd,ru) \ar@(lu,ld) & &
F:& \bullet \ar@/^1.5pc/[r] & \bullet \ar@/^1.5pc/[r] \ar@/^1.5pc/[l]& \bullet \ar@/^1.5pc/[l] 
}\]

\bigskip

 Following the definition of graph monoids (Definition~\ref{def:graphmonoid}), it is easy to see that $M_E \cong M_F$. Since the group completion of these monoids are the Grothendieck groups, we obtain $K_0(L_{\mathsf k}(E)) \cong K_0(L_{\mathsf k}(F)) \cong 0$. Since $L_{\mathsf k}(E)$ and $L_{\mathsf k}(F)$ are purely infinite simple (Theorem~\ref{drumoyne}), the combination of Theorem 6.3.38 and Theorem 6.3.32 of \cite{MR3729290} guarantees that $L_{\mathsf k}(E) \cong L_{\mathsf k}(F)$. However the period of the graph $E$ is $1$ whereas the period of $F$ is $2$. 
\end{example}

In contrast, graded isomorphism preserves the period of graphs: If $\phi\colon L_{\mathsf k}(E) \rightarrow L_{\mathsf k}(F)$ is a graded isomorphism, then 
\[
T_E \cong \mathcal V^{\gr}(L_{\mathsf k}(E))\cong \mathcal V^{\gr}(L_{\mathsf k}(F)) \cong T_F,
\] and by Theorem~\ref{thm:stronglyconnected} it follows that the period of $E$ and $F$ should be the same. 

As we mentioned in the beginning of the section, the ``primary colours'' of Leavitt path algebras are those given by line points, by cycles without exit, and by extreme cycles of graphs. These ``colours'' can be seen as the essential constituents of a Leavitt path algebra. We will now consider these colours in the language of talented monoids. 

\begin{lemma}\label{lem:idealgeneratedbylinepoint}
Suppose that the vertex $v$ is a line point in a graph $E$. Then the $\mathbb Z$-order ideal $\langle v\rangle$ is isomorphic to $\bigoplus_{\mathbb{Z}}\mathbb{N}$ as a $\mathbb{Z}$-monoid, where $\mathbb{Z}$ acts on $\bigoplus_{\mathbb{Z}}\mathbb{N}$ by right shifts -- i.e., as ${}^n(k_j)_{j\in\mathbb{Z}}=(k_{j-n})_{j\in\mathbb{Z}}$.
\end{lemma}
\begin{proof}
Let $x\in\langle v\rangle$, where $\langle v\rangle$ is the $\mathbb Z$-order ideal of $T_E$ generated by $v$ (see (\ref{idealorderandnotorder})). By the Confluence Lemma~\ref{confuu}, $x$ can be written as $x=\sum_j w_j(n_j)$, where $v$ flows to each $w_j$. Since $v$ is a line point, in $T_E$ we have
\[w_j(n_j)=w_{j-1}(n_j-1)=\cdots=v(n_j-d(n_j)),\]
which shows that, in fact, $x$ may be rewritten as $x=\sum_j k_j v(j)$ for certain $k_j\geq 0$. Moreover, this representation of $x$ is unique since $v$ is a line point, as we will now prove.

Suppose that $\alpha=\sum_j k_jv(j)$ and $\beta=\sum_j k'_j v(j)$ were two distinct representations of $x$ in $F_{\overline{E}}$. Since $T_E$ is cancellative, we can assume that $k_jk'_j=0$ for all $j$. We prove that all $k_j$ are equal to zero.

Suppose that this was not the case, and fix $j$ such that $k_j\neq 0$. By the Confluence Lemma~\ref{confuu}, there exists $y\in F_{\overline{E}}$ such that $\alpha,\beta\to y$. Write $y=\sum_i w_i(n_i)$. 

Since $k_j\neq 0$, then $v(j)$ appears in the representation of $\alpha$, so $v(j)$ flows to some $w_j(n_j)$. In particular, $v$ flows to $w_j$, and since there is only one path from $v$ to $w_j$ it follows that $n_j=j+d(w_j,v)$. But then, $w_j(n_j)$ must also be flowed into from some element in the representation of $\beta$, say $v(j')$. The same argument implies that $n_j=j'+d(w_j,v)$. Therefore $j=j'$. Thus $v(j)$ appears in the representation of $\beta$, so $k'_j\neq 0$, a contradiction.

Therefore the representation $x=\sum_j k_j v(j)$ is unique.

We may then unambiguously define $\phi\colon \langle v\rangle \to \bigoplus_\mathbb{Z}\mathbb{N}$ as $\phi(x)=\left(k_j(x)\right)_{j\in\mathbb{Z}}$, where the $k_j(x)$ are chosen such that $x=\sum_j k_j(x)v(j)$, for each $x\in\langle v\rangle$. Clearly, $\phi$ is an isomorphism of modules, and it is readily checked to preserve the $\mathbb{Z}$-actions.\qedhere
\end{proof}

The second ``colour'' of Leavitt path algebras is given by cycles without exits. The following lemma is also easy to verify, with similar arguments as in the proof of the one above (see also \cite[Example 2.4]{MR4040730}).

\begin{lemma}\label{lem:idealgeneratedbycyclewoexit}
Suppose that the vertex $v$ belongs to a cycle $c=e_1\cdots e_n$ without exit. Then $\langle v\rangle$ is isomorphic to $\bigoplus_{i=1}^n\mathbb{N}$ as a $\mathbb{Z}$-monoid, where $\mathbb{Z}$ acts on $\bigoplus_{\mathbb{Z}}\mathbb{N}$ as ${}^1(k_1,\ldots,k_n)=(k_n,k_1,\ldots,k_{n-1})$ (i.e., cyclically by a right shift).
\end{lemma}

Recall the notion of an essential ideal of a monoid (Definition~\ref{essentialll}). 

\begin{proposition}\label{prop:Hcofinaliffzorderidealessential}
    Let $E$ be a row-finite graph and $H\subseteq E^0$ a hereditary subset. Then the $\mathbb{Z}$-order ideal $\langle H\rangle$ generated by $H$ in $T_E$ is essential if, and only if, $H$ is cofinal in $E$, in the sense that every vertex of $E$ flows to some element of $H$.
\end{proposition}
\begin{proof}
    First we assume that $H$ is cofinal in $E$. In order to prove that $\langle H\rangle$ is essential, it suffices to prove that for every $v\in E^0$ there exists $x\in \langle H\rangle\setminus\left\{0\right\}$ with $x\leq v$. Since $H$ is cofinal, there exists a finite path $\alpha$ with $s(\alpha)=v$ and $r(\alpha)\in H$. Then the element $x={}^{|\alpha|}r(\alpha)$ belongs to $\langle H\rangle$ and $x\leq v$, as we wanted.
    
    Conversely, suppose that $\langle H\rangle$ is essential. Given $v\in E^0$, consider any nonzero $x\in\langle v\rangle\cap\langle H\rangle$. Then $x\leq \sum k_j v(j)$ and $x\leq\sum p_j h_j(n_j)$ for certain $k_j,p_j\geq 0$, $h_j\in H$ and $n_j\in\mathbb{Z}$. Repeated applications of the Confluence Lemma~\ref{confuu} and Proposition \ref{prop:flowpath} imply that there exists some vertex $u$ and some integer $i$ such that $x\geq u(i)$ and such that both $v(j)$ and $h_{j'}(n_{j'})$ flow to $u(i)$, for certain $j,j'\in\mathbb{Z}$. In particular, $v$ and $h_{j'}$ flow to $u$. Since $H$ is hereditary, $u$ belongs to $H$. Thus $v$ flows to some element of $H$. This proves the cofinality of $H$.\qedhere
\end{proof}

Let $E$ be a row-finite graph. We define $P_l(E)$ to be the set of line points of $E$; $P_c(E)$ the set of points which belong to cycles without exits; and $P_{ec}$ the set of points which belong to extreme cycles of $E$. Let $P_{lce}(E)$ be their union. The sets $P_l(E)$, $P_c(E)$ and $P_{ec}(E)$ are hereditary and pairwise disjoint, so the $\mathbb Z$-order ideal $\langle P_{lce}(E)\rangle$ of $T_E$ decomposes as a direct sum
\[\langle P_{lce}(E)\rangle=\langle P_l(E)\rangle\oplus\langle P_c(E)\rangle\oplus\langle P_{ec}(E)\rangle.\]

We can decompose $P_l(E)$, $P_c(E)$ and $P_{ec}(E)$ into ``minimal'' components. Define a relation $\sim$ on $E^0$ as $v\sim w$ if and only if $v$ and $w$ flow to a common vertex. The relation $\sim$ restricts to an equivalence relation on $P_{lce}(E)$, and each of the sets $P_l(E)$, $P_c(E)$ and $P_{ec}(E)$ is $\sim$-invariant -- i.e., if $x\sim y$ in $P_{lce}(E)$ then both $x$ and $y$ belong to the same of the sets $P_l(E)$, $P_c(E)$ or $P_{ec}(E)$. Equivalently on $P_{lce}(E)$, we have $x\sim y$ if and only if $\langle x\rangle=\langle y\rangle$.

If $A$ is a $\sim$-equivalence class of $P_{lce}(E)$, we have $\langle A\rangle=\langle a\rangle$ for any $a\in A$.

\begin{lemma}[{\cite[Lemma 3.7.10]{MR3729290}}]\label{lem:Plcecofinal}
    Let $E$ be a row-finite graph for which $E^0$ is finite. Then $P_{lce}(E)$ is cofinal.
\end{lemma}

Applying Lemmas \ref{lem:Plcecofinal}, \ref{lem:idealgeneratedbylinepoint}, \ref{lem:idealgeneratedbycyclewoexit} and Proposition \ref{prop:Hcofinaliffzorderidealessential}, we immediately conclude the characterization of the talented monoid of the ideal generated by the ``primary colours'' of a graph. This is an analogue of \cite[Theorem 3.7.9]{MR3729290}.

\begin{theorem}
    Let $E$ be a finite graph. Then the ideal $I_{lce}:=\langle P_{lce}(E)\rangle$ is essential in $T_E$, and it decomposes as a $\mathbb{Z}$-monoid as
    \[I_{lce}=\left(\bigoplus_{\alpha\in\Gamma_c}\left(\bigoplus_{\mathbb{Z}}\mathbb{N}\right)\right)\oplus\left(\bigoplus_{\beta\in\Gamma_l}\left(\bigoplus_{\#\beta}\mathbb{N}\right)\right)\oplus\left(\bigoplus_{\gamma\in\Gamma_{ec}}\langle c_\gamma\rangle\right),\]
    where
    \begin{itemize}
        \item $\Gamma_c$ is the set of $\sim$-equivalence classes contained in $P_c(E)$.
        \item $\Gamma_l$ is the set of $\sim$-equivalence classes contained in $P_l(E)$.
        \item $\Gamma_{ec}$ is the set of $\sim$-equivalence classes contained in $P_{ec}(E)$, and for each $\gamma\in\Gamma_{ec}$, $c_\gamma$ is any representative of $\gamma$.
    \end{itemize}
    Here, $\mathbb{Z}$ acts on $\bigoplus_{\mathbb{Z}}\mathbb{N}$ and on $\bigoplus_{\#\beta}\mathbb{N}$ by right shifts.
\end{theorem}

We can now use our results to give a finer description of the class of unital purely infinite simple Leavitt path algebras (compare with Theorem~\ref{drumoyne}).  These algebras are one the most interesting classes of graph algebras, which include Cuntz and Leavitt algebras. A characterisation of these algebras, in terms of the geometry of their associated  graphs, was one of the first to be obtained in the theory (\cite{MR2135030,MR3729290}).
Roughly, purely infinite simple algebras are associated to graphs which consists of a strongly connected component with all other vertices connecting to this component. This motivates the study of strongly connected graphs and algebraic objects attached to them.

The \emph{strongly connected component} of a finite graph $E$ without sinks is defined as the graph obtained by repeatedly removing all regular sources of $E$ until the graph has no sources. A variant of one direction of the next theorem was obtained by Pask and Rho \cite[Theorem~6.11]{MR2018239} in the setting of graph $C^*$-algebras. 

\begin{theorem}\label{hfgftrgfggf}
Let $E$ be a finite graph without sinks, $\mathsf{k}$ a field and $d\in\mathbb{N}$. The following are equivalent:
\begin{enumerate}
    \item $L_{\mathsf{k}}(E)$ is purely infinite simple and $L_{\mathsf{k}}(E)_0$ is a direct sum of $d$ minimal ideals. 
    \item The graph $E$ satisfies Condition (L), has a cycle, every vertex connects to every cycle and the strongly connected component of $E$ is of period $d$.
\end{enumerate}
\end{theorem}
\begin{proof}
(1) $\Rightarrow$ (2). Let $L_{\mathsf{k}}(E)$ be purely infinite simple. Recall from Theorem~\ref{drumoyne} the geometric properties of the graph $E$. Clearly $E$ does not have isolated vertices. Let $E'$ be the strongly connected component of $E$, obtained by repeatedly removing the sources from $E$. Clearly, $E'$ satisfies the same properties as $E$, and so $L_{\mathsf{k}}(E')$ is also purely infinite simple.

%By \cite[Proposition 3.2]{MR3045151} (see also \cite[Proposition 1.4]{MR2785945}), $L_{\mathsf{k}}(E')$ is Morita equivalent to $L_{\mathsf{k}}(E)$.%By \cite[Proposition 3.3]{MR0899719} (see also \cite[Proposition 21.11]{}), $L_{\mathsfk{k}}(E')$ is simple, and by \cite[]{}

By repeated applications of Proposition~\ref{valenjov}, we have $T_E \cong T_{E'}$ as $\mathbb{Z}$-monoids, so in particular $M_E\cong M_{E'}$. By Corollary \ref{thm:irredpropldecomposition}, $E'$ is strongly connected.

Since $E$ has no sink, then by \cite[Theorem~4]{MR3004584} $L_{\mathsf{k}}(E)$ is strongly graded and thus by Dade's theorem (\cite[\S 2.6]{MR3004584}) there is an equivalence of of categories
\[\Gr L_{\mathsf{k}}(E) \cong \Mod L_{\mathsf{k}}(E)_0.\]
This implies a monoid isomorphism $\mathcal{V}(L_{\mathsf{k}}(E)_0) \cong \mathcal{V}^{\gr}(L_{\mathsf k}(E))$. Putting these together we have
\[ T_{E'}\cong T_E \cong \mathcal V^{\gr} (L_{\mathsf{k}}(E))\cong \mathcal{V}(L_{\mathsf{k}}(E)_0).\]
Since $L_{\mathsf{k}}(E)_0$ is von Neumann unit-regular, there is a lattice isomorphism between the ideals of $L_{\mathsf{k}}(E)_0$ and the order ideals of $\mathcal{V}(L_{\mathsf{k}}(E)_0)$ (\cite[Corollary~15.21]{MR1150975}). Since $L_{\mathsf{k}}(E)_0$ is the direct sum of $d$ minimal ideals, this implies that $\mathcal{V}(L_{\mathsf{k}}(E)_0)$ and thus $T_{E'}$ is also the direct sum of $d$ simple order ideals. By Theorem \ref{thm:stronglyconnected} this implies that the period of $E'$ is $d$. 

\medskip

(2) $\Rightarrow$ (1) Note that the strongly connected component $E'$ of $E$ contains all the cycles of $E$ and, since $E$ is finite, every vertex in $E'$ can be flowed into from some vertex in a cycle. The conditions in (2) imply that $E'$ is indeed strongly connected.

Similar to the first part, the periodicity of $E'$ along with Theorem~\ref{thm:stronglyconnected} implies that $\mathcal{V}(L_{\mathsf{k}}(E)_0)$ can be written as a sum of $d$ simple order ideals and this implies that $\mathcal{L}(E)_0$ is a direct sum of $d$ minimal ideals.\qedhere
\end{proof}

Recall that a graph is called periodic if it is finite, strongly connected and has period $1$. Theorem~\ref{hfgftrgfggf} immediately gives the following corollary. A variant of this corollary was obtained in the setting of graph $C^*$-algebras by Pask and Rho (see \cite[Theorem~6.2]{MR2018239}). 

\begin{corollary}
Let $E$ be a finite graph and $\mathsf{k}$ a field. Then the strongly connected component of $E$ is aperiodic if and only if $L_\mathsf{k}(E)$ is purely infinite simple and $L_\mathsf{k}(E)_0$ is simple. 
\end{corollary}

\printbibliography

\end{document}